\documentclass[11pt]{amsart}
\usepackage[british]{babel}

\usepackage{amssymb,latexsym}
\usepackage{comment}
\usepackage{titletoc}
\usepackage{amsfonts, amsmath, amsthm, amssymb, fancybox, graphicx, color, times, babel}
\usepackage[utf8]{inputenc}
\usepackage{multirow}
\usepackage{enumerate}
\usepackage{stmaryrd}

\usepackage{xr-hyper}
\usepackage{hyperref}

\usepackage[usenames,dvipsnames]{xcolor}

\usepackage[all,cmtip]{xy}
\usepackage[normalem]{ulem}

\usepackage{tikz-cd}

\newtheorem {thm}{Theorem}
\newtheorem* {thm*}{Theorem}
\newtheorem {cor}[thm]{Corollary}
\newtheorem* {cor*}{Corollary}
\newtheorem {lem}[thm]{Lemma}
\newtheorem {prop}[thm]{Proposition}
\newtheorem* {prop*}{Proposition}
\theoremstyle{definition}
\newtheorem {rem}[thm]{Remark}
\newtheorem {defi}[thm]{Definition}
\newtheorem {exa}[thm]{Example}

\numberwithin{thm}{section}

\DeclareMathOperator{\Ann}{Ann}
\DeclareMathOperator{\Aut}{Aut}

\DeclareMathOperator{\End}{End}
\DeclareMathOperator{\Frac}{Frac}
\DeclareMathOperator{\Gal}{Gal}
\DeclareMathOperator{\Hom}{Hom}

\DeclareMathOperator{\GL}{GL}
\DeclareMathOperator{\SL}{SL}

\DeclareMathOperator{\Mat}{Mat}
\DeclareMathOperator{\id}{id}

\DeclareMathOperator{\tors}{tors}

\DeclareMathOperator{\kumm}{div}
\DeclareMathOperator{\im}{Im}

\DeclareMathOperator{\rk}{rk}

\DeclareMathOperator{\Ent}{Ent}

\newcommand{\Q}{\mathbb{Q}}
\newcommand{\R}{\mathbb{R}}
\newcommand{\NN}{\mathbb{N}}
\newcommand{\Z}{\mathbb{Z}}
\newcommand{\G}{\mathbb{G}}

\newcommand{\Kbar}{\overline{K}}
\newcommand{\Abar}{\overline{A}}

\newcommand{\set}[1]{\left\{ #1 \right\}}

\renewcommand{\geq}{\geqslant}
\renewcommand{\leq}{\leqslant}

\begin{document}

\title{Radical Entanglement for Elliptic Curves}
\author{Sebastiano Tronto}
\address[]{Department of Mathematics, University of Luxembourg, 6 av.\@ de la Fonte, 4364 Esch-sur-Alzette, Luxembourg}
\email{sebastiano.tronto@uni.lu}

\begin{abstract}
  Let $G$ be a commutative connected algebraic group over a number field $K$,
  let $A$ be a finitely generated and torsion-free subgroup of $G(K)$ of
  rank $r>0$ and, for $n>1$, let $K\left(n^{-1}A\right)$ be the smallest
  extension of $K$ inside an algebraic closure $\Kbar$ over which all the
  points $P\in G(\Kbar)$ such that $nP\in A$ are defined. We denote by $s$
  the unique non-negative integer such that $G(\overline K)[n]\cong (\Z/n\Z)^s$
  for all $n\geq 1$. We prove that, under certain conditions, the ratio between
  $n^{rs}$ and the degree $\left[K\left(n^{-1}A\right):K(G[n])\right]$ is
  bounded independently of $n> 1$ by a constant that depends only on the
  $\ell$-adic Galois representations associated with $G$ and on some arithmetic
  properties of $A$ as a subgroup of $G(K)/G(K)_{\tors}$. In particular we 
  extend the main theorems of \cite{lt} about elliptic curves to the case of
  arbitrary rank.
\end{abstract}

\maketitle

\section{Introduction}

\subsection{Setting}

Let $K$ be a number field and fix an algebraic closure $\overline{K}$ of $K$.
If $G$ is a commutative connected algebraic group over $K$ and $A$ is a finitely generated and torsion-free subgroup of $G(K)$, for any positive integer $n$ we may consider the field $K\left(n^{-1}A\right)$, that is the smallest extension of $K$ inside $\Kbar$ containing the coordinates of all points $P\in G(\Kbar)$ such that $nP\in A$. This is a Galois extension of $K$ containing the $n$-th torsion field $K(G[n])$ of $G$.

If $G=\mathbb{G}_m$ is the multiplicative group, such extensions are studied by classical Kummer theory. The more general case of an extension of an abelian variety by a torus is treated in Ribet's foundational paper \cite{ribet}. Under certain assumptions, for example if $G$ is the product of an abelian variety and a torus and $A$ has rank $1$, it is known that the ratio
\begin{align}
\label{eqn:firstRatio}
  \frac{n^{s}}{\left[K\left(n^{-1}A\right):K(G[n])\right]}
\end{align}
where $s$ is the unique positive integer such that $G(\Kbar)[n]\cong (\Z/n\Z)^s$
for all $n\geq 1$,
is bounded independently of $n$ (see also \cite[Th\'eor\`eme 5.2]{bertrand} and \cite[Lemme 14]{hindry}).

In \cite{lt} Lombardo and the author were able to give an effective bound for the ratio \eqref{eqn:firstRatio} if $G=E$ is an elliptic curve with $\End_K(E)=\Z$ and $A=\langle \alpha\rangle$ has rank $1$. Moreover, a uniform bound in the case $K=\Q$, under some necessary assumptions on the divisibility of $\alpha$ in $E(K)/E(K)_{\tors}$, was given.

The bounds given in \cite{lt} essentially depend on three properties of $E$ and $\alpha$:
\begin{enumerate}[(1)]
\item The finitess of the divisibility of $\alpha$ in $E(K)/E(K)_{\tors}$;
\item Properties of the $\ell$-adic Galois representations associated with $E$, for every prime $\ell$;
\item The finiteness of the exponent of $H^1(\Gal(K(E(\Kbar)_{\tors})\mid K),E(\Kbar)_{\tors})$.
\end{enumerate}

The goal of the present paper is twofold: firstly, we use the properties of $r$-extensions of abelian groups introduced by Palenstijn in \cite{pmaster} and \cite{palenstijn} to generalize the methods of \cite{lt} to groups $A$ of arbitrary finite rank and any commutative connected algebraic group $G$ that satisfies the same properties mentioned above. 
The result we obtain is the following (see Theorem \ref{thm:GeneralBoundEntanglement}):

\begin{thm}
\label{thm:main1intro}
Let $G$ be a commutative connected algebraic group over a number field $K$ and let $A\subseteq G(K)$ be a finitely generated and torsion-free subgroup of rank $r>0$. Let $s$ be the unique non-negative integer such that $G[n]\cong (\Z/n\Z)^s$ for all $n\geq 1$. Let $H$ denote, after a choice of basis, the image of the adelic Galois representation associated with $G$ over $K$
\begin{align*}
\Gal(\Kbar\mid K)\to\GL_s(\hat\Z).
\end{align*}
For every prime $\ell$, let $H_\ell$ denote the image of $H$ under the projection $\GL_s(\hat\Z)\to\GL_s(\Z_\ell)$ and denote by $\Z_\ell[H_\ell]$ the closed $\Z_\ell$-subalgebra of $\Mat_{s\times s}(\Z_\ell)$ generated by $H_\ell$. Assume that
\begin{enumerate}[(1)]
\item There is an integer $d_A\geq 1$ such that
\begin{align*}
d_A\cdot\{P\in G(K)\mid \exists n\in\NN_{\geq 1}:\,nP\in A\}\subseteq A+G(K)_{\tors}\,.
\end{align*}
\item There is an integer $N\geq 1$ such that $\Z_\ell[H_\ell]\supseteq N\Mat_{s\times s}(\Z_\ell)$ for every prime $\ell$.
\item There is an integer $M\geq 1$ such that the exponent of $H^1(\Gal(K_\infty\mid K),G(\Kbar)_{\tors})$ divides $M$, where $K_{\infty}=K(G(\Kbar)_{\tors})$.
\end{enumerate}
Then for every $n\geq 1$ the ratio
\begin{align*}
\frac{n^{rs}}{\left[K\left(n^{-1}A\right):K(G[n])\right]}
\end{align*}
divides $(d_ANM)^{rs}$.
\end{thm}

The first condition of Theorem \ref{thm:main1intro} is always satisfied if $G$ is an abelian variety or $G=\G_m$ (see Example \ref{exa:divisibilityParam}). We call such an integer $d_A$ a \emph{divisibility parameter} for $A$ in $G(K)$.  One has $d_A=1$ if, for example, the group $G(K)$ is finitely generated and torsion-free and $A=G(K)$.

Notice that if a set of generators for $A$ is known, modulo the torsion
subgroup of $G(K)$, in terms of a $\Z$-basis of
$G(K)/G(K)_{\tors}$, one can compute a divisibility parameter $d_A$. See section
\ref{subsec:dParam}.

Our second goal is to apply Theorem \ref{thm:main1intro} to some specific cases.
In particular, we generalize the results of \cite{lt} to the case of arbitrary rank. Theorems \ref{thm:main2intro} and \ref{thm:main3intro} below follow from Theorems \ref{thm:mainEC1}, \ref{thm:mainEC1cm} and \ref{thm:mainEC2} and Lemma \ref{lemma:entBoundsDegrees}.

\begin{thm}
\label{thm:main2intro}
Let $E$ be an elliptic curve over a number field $K$ such that $\End_K(E)=\Z$. Let $A$ be a finitely generated and torsion-free subgroup of $E(K)$ of rank $r>0$. 
There is an effectively computable integer $N>1$, depending only on $E$ and $K$, such that for every $n\geq 1$
\begin{align*}
\frac{n^{2r}}{\left[K\left(n^{-1}A\right):K(E[n])\right]}\quad \text{divides} \quad \left(d_AN\right)^{2r}
\end{align*}
where $d_A$ is a divisibility parameter for $A$ in $E(K)$.
\end{thm}

\begin{thm}
\label{thm:main3intro}
There is a universal constant $C\geq 1$ such that for every elliptic curve $E$ over $\Q$, for every torsion-free subgroup $A$ of $E(\Q)$ and for every $n\geq 1$
\begin{align*}
  \frac{n^{2\rk(A)}}{\left[\Q\left(n^{-1}A\right):\Q(E[n])\right]}\quad \text{divides} \quad \left(d_AC\right)^{2\rk(A)}
\end{align*}
where $d_A$ is a divisibility parameter for $A$ in $E(\Q)$.
\end{thm}

\subsection{Notation}

If $A$ is an abelian group and $n$ is a positive integer we denote by $A[n]$ the subgroup of the elements of $A$ of order dividing $n$. We denote by $A_{\tors}$ the subgroup consisting of all elements of $A$ of finite order. We denote by $\rk(A)$ the \emph{rank} of $A$, that is the dimension of $A\otimes_{\Z}\Q$ as a $\Q$-vector space.

If $R$ is a commutative ring, then we denote by $\Mat_{n\times m}(R)$ the $R$-module of $n\times m$ matrices with entries in $R$, which we regard as an $R$-algebra if $n=m$.
If at least one between $n$ and $m$ is zero then $\Mat_{n\times m}(R)$ is the trivial ring (or trivial $R$-algebra if $n=m=0$).
For $n>0$ we denote by $\GL_n(R)$ the group of invertible $n\times n$ matrices with entries in $R$.

For any prime number $\ell$ and any non-zero integer $n$ we denote by $v_\ell(n)$ the $\ell$-adic valuation of $n$. We denote by $\Z_\ell$ the ring of $\ell$-adic integers and by $\hat\Z$ the ring of profinite integers, which we identify with the product $\prod_\ell\Z_\ell$.

If $K$ is a number field and $\Kbar$ is a fixed algebraic closure of $K$, we denote by $\zeta_n$ a primitive $n$-th root of unity in $\Kbar$, for any positive integer $n$. If $G$ is any algebraic group over $K$ and $L$ is any field extension of $K$, we denote by $G(L)$ the group of $L$-points of $G$. If $S$ is a subset of $G(\Kbar)$, we denote by $K(S)$ the subfield of $\Kbar$ whose elements are fixed by
\begin{align*}
H=\left\{g\in \Gal(\Kbar\mid K)\mid g(P)=P\quad\forall\, P\in S\right\}.
\end{align*}
If $G$ is embedded in an affine or projective space (notice that, as a
consequence of Chevalley's structure theorem, any algebraic group over a
field is quasi-projective) then $K(S)$ coincides with the field generated by
$K$ and any choice of affine coordinates of all points $P\in S$.

\subsection{Structure of the paper}

After some necessary group-theoretic preliminaries in Section~ \ref{sec:preliminaries}, we investigate in Section \ref{sec:sExtensions} the theory of $s$-extensions of abelian groups introduced by Palenstijn. Much of the content of that section can be found, with little differences, in \cite{pmaster}.

We then move on to prove some $\hat\Z$-linear algebra results in Section \ref{sec:LinearAlgebra}, and finally develop our theory of entanglement for commutative algebraic groups in Section \ref{sec:GeneralTheory}. In Section \ref{sec:ellipticNoCM} we apply this theory to the case of elliptic curves without complex multiplication.

\subsection{Acknowledgements}
I am grateful to my advisors Antonella Perucca and Peter Bruin for their constant support during the preparation of this paper. I am also very grateful to Hendrik Lenstra and Peter Stevenhagen for giving me some of the main ideas for this work.

\section{Group-theoretic preliminaries}
\label{sec:preliminaries}

We collect here some basic group-theoretic results that we will need throughout this paper.

\subsection{Pontryagin duality}

Let $G$ be a locally compact Hausdorff topological abelian group. Let $S^1=\R/\Z$ with the usual topology. The group  $\Hom(G,S^1)$ of continuous homomorphisms from $G$ to $S^1$ endowed with the compact-open topology is itself a locally compact abelian group, and it is called the \emph{group of characters} or the \emph{(Pontryagin) dual} of $G$ (see \cite[Chapter~6]{pontryagin}). We will denote it by $G^\wedge$.

\begin{exa}
\label{exa:dualOfQmodZ}
Consider $\Q/\Z$ as a topological group with the discrete topology. We have $(\Q/\Z)^\wedge\cong \hat\Z$. To see this, notice first that for every positive integer $n$ there is a natural isomorphism
\begin{align*}
\Hom\left(\frac{\frac{1}{n}\Z}{\Z},\Q/\Z\right)\cong \Z/n\Z
\end{align*}
given by sending a homomorphism $\varphi:\frac{1}{n}\Z/\Z\to \Q/\Z$ to the unique $d\in \Z/n\Z$ such that $\varphi\left(\frac{1}{n}\right)=\frac{d}{n}$.
Now we have
\begin{align*}
\Hom(\Q/\Z,S^1)&=\Hom(\Q/\Z,\Q/\Z)\cong\\
&\cong \Hom\left(\varinjlim_n\frac{\frac{1}{n}\Z}{\Z},\Q/\Z\right)\cong\\
&\cong \varprojlim_n\Hom\left(\frac{\frac{1}{n}\Z}{\Z},\Q/\Z\right)\cong\\
&\cong \varprojlim_n \Z/n\Z.
\end{align*}
The maps forming this last projective system are just the natural projections, since for $n\mid m$ the restriction of 
\begin{align*}
\varphi:\Z/m\Z&\to\Q/\Z\\
\frac{1}{m}&\mapsto \frac{d}{m}
\end{align*}
to $\Z/n\Z$ maps $\frac{1}{n}$ to $\frac{d}{n}$. So we get $\Hom(\Q/\Z,S^1)\cong \hat \Z.$
\end{exa}

\begin{rem}
\label{rem:HigherDimPontryagin}
In Section \ref{sec:LinearAlgebra} we will need a higher-dimensional analogue of Example \ref{exa:dualOfQmodZ}. By the previous example we easily deduce that, for $r,s\geq 1$, the group $\Hom((\Q/\Z)^r,(\Q/\Z)^s)$ can be identified with $\Mat_{s\times r}(\hat\Z)$. This can be seen directly on the finite level as follows: let
\begin{align*}
\begin{array}{cccc}
\varphi:&\left(\dfrac{\frac{1}{n}\Z}{\Z}\right)^r&\to&\left(\dfrac{\frac{1}{n}\Z}{\Z}\right)^s\\
& \left(\frac{1}{n},0,\dots,0\right)&\mapsto& \left(\frac{d_{11}}{n},\frac{d_{21}}{n},\dots,\frac{d_{s1}}{n}\right)\\
& \left(0,\frac{1}{n},\dots,0\right)&\mapsto& \left(\frac{d_{12}}{n},\frac{d_{22}}{n},\dots,\frac{d_{s2}}{n}\right)\\
& \vdots & & \vdots\\
& \left(0,0,\dots,\frac{1}{n}\right)&\mapsto& \left(\frac{d_{1r}}{n},\frac{d_{2r}}{n},\dots,\frac{d_{sr}}{n}\right)
\end{array}
\end{align*}
be a group homomorphism. The matrix $D_\varphi=(d_{ij})\in \Mat_{s\times r}(\Z/n\Z)$ completely describes the homomorphism $\varphi$, and the map $\varphi\mapsto D_\varphi$ is easily checked to be a group isomorphism between $\Hom((\frac{1}{n}\Z/\Z)^r,(\frac{1}{n}\Z/\Z)^s)$ and $\Mat_{s\times r}(\Z/n\Z)$. Passing to the limit in $n$ we obtain a description of the natural isomorphism $\Hom((\Q/\Z)^r,(\Q/\Z)^s)\cong \Mat_{s\times r}(\hat\Z)$.

Furthermore, if $r=s$ the map $\varphi\mapsto D_\varphi$ is a ring homomorphism from $\End((\Q/\Z)^s)$ to $\Mat_{s\times s}(\hat\Z)$. This allows us to identify $\Aut((\Q/\Z)^s)=\End((\Q/\Z)^s)^{\times}$ with $\GL_s(\hat\Z)$.
\end{rem}

\begin{thm}[{Pontryagin duality, see {\cite[Theorems 39 and 40]{pontryagin}}}]
\label{thm:Pontryagin}
The functor $\Hom(-,S^1)$ that maps $G$ to its dual $G^\wedge$ is an anti-equivalence of the category of locally compact Hausdorff topological abelian groups with itself. Moreover $(G^\wedge)^\wedge$ is naturally isomorphic to $G$.

This anti-equivalence induces an inclusion-reversing bijection between the closed subgroups of any locally compact topological abelian group $G$ and those of $G^\wedge$, given by
\begin{align*}
\begin{array}{ccc}
  \set{\text{closed subgroups of }G} & \longleftrightarrow & \set{\text{closed subgroups of }G^\wedge}\\[0.4em]
  U &\longmapsto & \Ann U:=\set{f\in G^\wedge\mid f(u)=0\,\forall\, u\in U}\\[0.4em]
\set{g\in G\mid f(g)=0\,\forall f\in V}=:\Ann V& \longmapsfrom & V
\end{array}
\end{align*}
Moreover, $G$ is discrete if and only if $G^\wedge$ is compact, and $G$ is discrete and torsion if and only if $G^\wedge$ is profinite.
\end{thm}

\subsection{Relative automorphism groups}

In this section we establish some basic results on relative automorphism groups of abelian groups, that is the groups containing those automorphisms that restrict to the identity on a given subgroup. 

If $A$ is an abelian group and $B,C$ are abelian groups containing $A$ as a subgroup, then we denote by $\Hom_A(B,C)$ the set of homomorphisms $B\to C$ that restrict to the identity on $A$. Similarly we define the ring of endomorphisms $\End_A(B)$. We also denote by $\Aut_A(B)$ the group of all automorphisms of $B$ that restrict to the identity on $A$. We call any element of $\Aut_A(B)$ an $A$-automorphism of $B$.

\begin{lem}
\label{lem:pushoutMap}
Let $M$ and $N$ be abelian groups and let $A$ and $B$ be subgroups of $M$. If $f:A\to N$ and $g:B\to N$ are group homomorphisms such that $f_{|A\cap B}=g_{|A\cap B}$, then there exists a unique map $\varphi:A+B\to N$ such that $\varphi_{|A}=f$ and $\varphi_{|B}=g$.
\end{lem}
\begin{proof}
This is just a rephrasing of the universal property of $A+B$ as the pushout of $A\cap B\hookrightarrow A$ and $A\cap B\hookrightarrow B$.
\end{proof}

\begin{defi}
\label{def:normal}
Let $A\subseteq B\subseteq M$ be abelian groups. We say that $B$ is \emph{$A$-normal in $M$} if the restriction to $B$ of every element of $\Aut_A(M)$ maps $B$ surjectively to itself.

If $B'\subseteq M$ is a subgroup not necessarily containing $A$, then we say that $B'$ is \emph{$A$-normal in $M$} if the following two conditions hold:
\begin{enumerate}[(1)]
\item The group $B'$ is $(A\cap B')$-normal in $A+B'$ and
\item The group $A+B'$ is $A$-normal in $M$.
\end{enumerate}
\end{defi}

\begin{rem}
The choice of the word \emph{normal} in the above definition is in analogy with the case of field extensions in Galois theory.
\end{rem}

\begin{rem}
\label{rem:NormalOverBiggerBaseGroup}
Let $A\subseteq B\subseteq C\subseteq M$ be abelian groups. If $C$ is $A$-normal in $M$, then $C$ is also $B$-normal in $M$. If $B$ is $A$-normal in $C$ and $C$ is $A$-normal in $M$, then $B$ is $A$-normal in $M$.
\end{rem}

If $A\subseteq B\subseteq M$ are abelian groups, then $B$ is $A$-normal in $M$ if and only if the restriction map $\Aut_A(M)\to \Hom_A(B,M)$ factors via $\Aut_A(B)$. In this situation we call this map $\Aut_A(M)\to \Aut_A(B)$ the \emph{natural restriction map}.

\begin{lem}
\label{lem:autFixIntersection}
Let $M$ be an abelian group and let $A,B\subseteq M$ be subgroups of $M$. Assume that $B$ is $A$-normal in $A+B$.
Then the natural restriction map $\Aut_{A\cap B}(A+B)\to\Aut_{A\cap B}(B)$ induces an isomorphism $\Aut_A(A+B)\cong\Aut_{A\cap B}(B)$. 
\end{lem}
\begin{proof}
The inclusion $\Aut_A(A+B)\hookrightarrow\Aut_{A\cap B}(A+B)$ composed with the natural restriction yields a group homomorphism $\rho:\Aut_A(A+B)\to \Aut_{A\cap B}(B)$, which is injective because $\ker \rho=\Aut_{A+B}(A+B)=1$.

Let $\sigma\in \Aut_{A\cap B}(B)$ and let $\tilde\sigma:A+B\to A+B$ be the homomorphism obtained by applying Lemma \ref{lem:pushoutMap} to $\sigma$ and $\id_A$. This map is clearly surjective, since every element of $A$ and every element of $B$ are in its image. If $\tilde\sigma(a+b)=0$ for some $a\in A$ and some $b\in B$, then $\sigma(b)=-a\in A\cap B$, which implies that $b\in A\cap B$ and thus $a+b=0$. So $\tilde\sigma$ is injective, thus an automorphism. We conclude that $\rho$ is an isomorphism.
\end{proof}

\subsection{Projective limits of exact sequences}

\begin{rem}
\label{rem:exactSequenceAction}
Let 
\begin{align*}
1\to A\to G\to H\to 1
\end{align*}
be an exact sequence of groups, and assume that $A$ is abelian. Then there is a natural left action of $H$ on $A$, defined as follows.

Let $h\in H$ and consider any lift $\tilde h\in G$ of $h$. Then the action of $h$ on $a\in A$ is defined as
\begin{align*}
\tilde{h}a\tilde h^{-1}
\end{align*}
where we see $a$ as an element of $G$ via the inclusion map. This definition does not depend on the choice of the lift $\tilde h$, because if $\hat h$ is a different lift of $h$ then $\hat h = \tilde h b$ for some $b\in A$, and we have $\hat h a\hat h^{-1}=\tilde hbab^{-1}\tilde h^{-1}=\tilde ha\tilde h^{-1}$. Moreover we have that $\tilde{h}a\tilde{h}^{-1}$ is mapped to $1$ in $H$, so this clearly defines an action of $H$ on $A$.
\end{rem}

The following result is fairly standard, so we state it without proof.

\begin{lem}
\label{lem:generalTopologyExactSequenceLimit}
Let $I$ be a partially ordered set. For every $i\in I$ let $\mathcal{A}_i$ denote an exact sequence of profinite topological groups
\begin{align*}
1\to A_i'\to A_i \to A_i''\to 1
\end{align*}
such that $A_i'$ and $A_i''$ have the subspace and quotient topology with respect to $A_i$, respectively. For every $i\leq j$ let $\rho_{ij}:\mathcal{A}_j\to \mathcal{A}_i$ be a map of exact sequences such that $\{(\mathcal{A})_{i\in I},(\rho_{ij})_{i,j\in I}\}$ is a projective system. Let $\{\mathcal{A},(\pi_i)_{i \in I}\}$ be the limit of this projective system, where $\mathcal{A}$ is
\begin{align*}
1\to A'\to A\to A''\to 1\,.
\end{align*}
Then the subspace topology on $A'$ and the quotient topology on $A''$ coincide with their respective limit topology.
\end{lem}

\section{\texorpdfstring{$s$}{s}-extensions of abelian groups}
\label{sec:sExtensions}

In this section we are going to revisit the theory of certain kinds of extensions of abelian groups that were first introduced by Palenstijn in his master thesis \cite{pmaster}.
These extensions arise naturally when considering the so-called \emph{division
points} of a certain subgroup $A$ of the rational points of a commutative
algebraic group. In particular, the automorphism groups of these extension
provide a framework to study the Galois groups of field extensions generated by
division points.

\subsection{General definitions and first results}

Fix a positive integer $s$.

\begin{defi}
\label{def:sExt}
Let $A$ be a finitely generated abelian group. An \emph{$s$-extension} of $A$ is an abelian group $B$ containing $A$ such that:
\begin{enumerate}[(1)]
\item $B/A$ is torsion;
\item the torsion subgroup of $B$ is isomorphic to a subgroup of $(\Q/\Z)^s$.
\end{enumerate}
\end{defi}

\begin{rem}
\label{rem:necCondAdmitsSExtension}
A necessary (and sufficient) condition for a finitely generated abelian group $A$ to admit an $s$-extension is that the torsion subgroup $A_{\tors}$ of $A$ can be embedded in $(\Q/\Z)^s$.
\end{rem}

\begin{defi}
Let $A$ be a finitely generated abelian group. For every $s$-extension $B$ of $A$, every $a\in A$ and every positive integer $n$ we call any $b\in B$ such that $nb=a$ an \emph{$n$-division point} of $a$ (in $B$). We denote by
\begin{align*}
n_B^{-1}a:=\set{b\in B\mid nb=a}
\end{align*}
the set of $n$-division points of $a$. We omit the subscript $B$ from $n_B^{-1}$ if this is clear from the context.
We also denote by
\begin{align*}
B_n:=\set{b\in B\mid nb\in A}=\bigcup_{a\in A}n_B^{-1}a
\end{align*}
the set of all $n$-division points of elements of $A$, which is again an $s$-extension of $A$. Notice that for $n\mid m$ we have $B_n\subseteq B_m$ and that $B=\bigcup_{n\geq 1}B_n$. 
\end{defi}

\begin{rem}
\label{rem:nDivisionPointsInBijectionWithNtorsion}
Assume that $n^{-1}a$ is not empty. For any fixed $b_0\in n_B^{-1}a$, the map
\begin{align*}
\begin{array}{ccc}
n_B^{-1}a & \to & B[n]\\
b & \mapsto & b-b_0
\end{array}
\end{align*}
is a bijection.
\end{rem}

The following lemmas will be used in what follows, in particular in Section \ref{subsec:Aut}.

\begin{lem}
\label{lem:preservesDivisionPoints}
Let $B$ and $C$ be two $s$-extensions of a finitely generated abelian group $A$ and let $\varphi:B\to C$ be a group homomorphism that is the identity on $A$. For every $a\in A$ and every $b\in n_B^{-1}a$ we have $\varphi(b)\in n_C^{-1}a$. In particular, we have $\varphi(B_n)\subseteq C_n$.
\end{lem}
\begin{proof}
It is enough to notice that $n\varphi(b)=\varphi(nb)=\varphi(a)=a$.
\end{proof}

\begin{lem}
\label{lem:InjOnTors}
Let $B$ and $C$ be two $s$-extensions of a finitely generated abelian group $A$ and let $\varphi:B\to C$ be a group homomorphism that is the identity on $A$. The kernel of $\varphi$ is contained in $B_{\tors}$. Moreover, if for every prime $\ell$ the restriction of $\varphi$ to $B[\ell]$ is injective, then $\varphi$ is injective.
\end{lem}
\begin{proof}
Let $b\in\ker\varphi$ and let $n$ be a positive integer such that $nb=a\in A$. By Lemma \ref{lem:preservesDivisionPoints} we have $0\in n^{-1}_Ca$, which implies that $a=0$. In particular, $b$ is torsion.
For the second assertion, assume that $b\neq 0$ and let $\ell$ be a prime dividing the order of $b$. But then $b$ has a multiple of order $\ell$ which is in $\ker\varphi$, a contradiction.
\end{proof}

\begin{lem}
\label{lem:InjSurj}
Let $B$ be an $s$-extension of a finitely generated abelian group $A$ and let
$\varphi:B\to B$ be an endomorphism that is the identity on $A$.
If $\varphi$ is injective, then it is an automorphism.
\end{lem}
\begin{proof}
Assume first that $\varphi$ is injective and let $b\in B$. Let $n$ be a positive
integer such that $nb=a\in A$. By Lemma \ref{lem:preservesDivisionPoints}
we have $\varphi(n^{-1}a)\subseteq n^{-1}a$. Since $n^{-1}a$ is finite there must
be some $b'\in n^{-1}a$ such that $\varphi(b')=b$, hence $\varphi$ is surjective.
\end{proof}

The following proposition gives a criterion to verify if an $s$-extension is normal in the sense of Definition \ref{def:normal}.
\begin{prop}
\label{prop:NormalityCriterion}
Let $B$ be an $s$-extension of a finitely generated abelian group $A$ and let $C\subseteq B$ be a subgroup. If $\Hom_{A\cap C}(C,B)\subseteq\Hom_{A\cap C}(C,C)$, then $C$ is $A$-normal in $B$.

Moreover, under the same assumptions, for every $A\subseteq A'\subseteq C\subseteq B'\subseteq B$ we have that $C$ is $A'$-normal in $B'$.
\end{prop}
\begin{proof}
First of all, notice that $C$ is an $s$-extension of $A\cap C$ and that $A+C$ is an $s$-extension of $A$. Let now $\sigma\in \Aut_{A\cap C}(A+C)$ and consider its restriction $\sigma_C:C\to A+C$. We then have
\begin{align*}
\sigma_C\in \Hom_{A\cap C}(C,A+C)\subseteq \Hom_{A\cap C}(C,B)\subseteq\Hom_{A\cap C}(C,C).
\end{align*}
Moreover $\sigma_C$ is injective, thus an automorphism by Lemma \ref{lem:InjSurj}. This shows that $C$ is $(A\cap C)$-normal in $A+C$.

To see that $A+C$ is $A$-normal in $B$, let $\tau\in \Aut_A(B)$ and consider its restriction $\tau_{A+C}:A+C\to B$. Since $\tau$ is the identity on $A$ and the image of its restriction to $C$ is contained in $C$ by assumption, we have that the image of $\tau_{A+C}$ is contained in $A+C$. Since $\tau$ is injective, by applying Lemma \ref{lem:InjSurj} we see that $\tau_{A+C}$ is an $A$-automorphism of $A+C$, so we conclude that $A+C$ is $A$-normal in $B$. Thus $C$ is $A$-normal in $B$.

The second assertion follows from the first by noticing that $\Hom_{A'\cap C}(C,B')$ is contained in $\Hom_{A\cap C}(C,B)$.
\end{proof}

\begin{exa}
\label{exa:criterion}
Let $B$ be an $s$-extension of a finitely generated abelian group $A$. Proposition \ref{prop:NormalityCriterion} can be applied in the following cases:
\begin{enumerate}[(1)]
\item Let $C$ be either $B_{\tors}$ or $B[n]$ for some positive integer $n$. Then 
  the image of every group homomorphism from $C$ to $B$ is contained in $C$, so
  in particular $\Hom_{A\cap C}(C,B)\subseteq \Hom_{A\cap C}(C,C)$.
\item If $C=B_n$ for some positive integer $n$, then by Lemma \ref{lem:preservesDivisionPoints} we have $\Hom_A(B_n,B)\subseteq \Hom_A(B_n,B_n)$ and hence $\Hom_{A\cap B_n}(B_n,B)\subseteq \Hom_{A\cap B_n}(B_n,B_n)$.
\end{enumerate}
\end{exa}

\subsection{Automorphisms of $s$-extensions}
\label{subsec:Aut}

We now study the automorphisms of an $s$-extension that are the identity on the base group.
Recall that if $B$ is an abelian group and $A\subseteq B$ is a subgroup we denote by $\Aut_A(B)$ the group of all automorphisms of $B$ that restrict to the identity on $A$.

Fix for the remainder of this section a finitely generated abelian group $A$.

The following result is a generalization of \cite[Lemma 1.8]{palenstijn}, and the proof is essentially the same. We include it here for the sake of completeness.

\begin{prop}
\label{prop:ResIsSurj}
Let $B$ be an $s$-extension of $A$ and let $C\subseteq B$ be a subgroup. If $C$ is $A$-normal in $B$, the image of the restriction map $\Aut_A(B)\to\Hom_{A\cap C}(C,B)$ is $\Aut_{A\cap C}(C)$.
\end{prop}
\begin{proof}
By Lemma \ref{lem:autFixIntersection} we have $\Aut_A(A+C)\cong\Aut_{A\cap C}(C)$ via the restriction map, so it is enough to show that the restriction $\Aut_A(B)\to \Aut_A(A+C)$, which exists because $A+C$ is $A$-normal in $B$, is surjective. Thus we may assume that $A\subseteq C$.

In view of Lemma \ref{lem:InjSurj} it is enough to prove that every $\varphi\in \Aut(C)$ can be extended to an injective homomorphism $B\to B$. Consider the set of pairs $(M,\phi)$, where $M$ is a subgroup of $B$ containing $C$ and $\phi:M\to B$ is an injective homomorphism extending $\varphi$, ordered by inclusion
\begin{align*}
(M,\phi)\subseteq (M',\phi') \qquad \iff \qquad M\subseteq M' \quad \text{and} \quad \phi_{|M}'=\phi.
\end{align*}
By Zorn's Lemma this ordered set admits a maximal element $(\tilde B,\tilde\varphi)$ and we need to show that $\tilde B=B$.
We prove this by contradiction, assuming that there exists $x\in B\setminus \tilde B$ and proving that we can then extend $\tilde{\varphi}$ to an injective map $\langle \tilde B,x\rangle\to B$.

Assume first that the order of $x$ is a prime number $\ell$. An element of $\tilde B$ mapping to $B[\ell]$ must be in $\tilde B[\ell]$ because $\tilde\varphi$ is injective.
Since $x\in B[\ell]\setminus \tilde B[\ell]$ we have $\#\tilde B[\ell]< \#B[\ell]$, so there must be $y\in B[\ell]\setminus \{0\}$ that is not in the image of $\tilde\varphi$. Using Lemma \ref{lem:pushoutMap} we can then extend $\tilde\varphi$ to $\langle \tilde B,x\rangle$ by letting $\tilde\varphi(x):=y$. The map we obtain is still injective, so we may assume that $\tilde B$ contains all elements of prime order of $B$.

Let now $k$ be the smallest positive integer such that $kx\in \tilde B$. Up to replacing $x$ with a suitable multiple, we may assume that $k=\ell$ is a prime number. Let $b=\ell x\in  \tilde B$. The fact that  $B[\ell]\subseteq \tilde B$ implies that $\ell^{-1}_Bb\subseteq B\setminus \tilde B$.

Consider now $\tilde\varphi(b)\in B$ and let $y\in \ell_B^{-1}\tilde\varphi(b)$. If $y\in \im(\tilde\varphi)$, then there is $z\in\tilde B$ such that $\tilde{\varphi}(z)=y$, thus $\tilde{\varphi}(\ell z)=\ell y=\tilde\varphi(b)$ and so $\ell z=b$, a contradiction. Since $\tilde B\cap \langle x\rangle=\langle \ell x\rangle$ and $\tilde\varphi(\ell x)=\ell y$, using again Lemma \ref{lem:pushoutMap} we can extend $\tilde \varphi$ to $\langle \tilde B,x\rangle$ by letting $\tilde \varphi(x):=y$. By Lemma \ref{lem:InjOnTors}, the homomorphism $\langle \tilde B,x\rangle\to B$ that we obtain is still injective.

We conclude that $\tilde B=B$, thus the restriction map $\Aut_A(B)\to \Aut_A(C)$ is surjective.
\end{proof}

\begin{prop}
\label{prop:AutIsHom}
Let $B$ be an $s$-extension of $A$.
There is a canonical isomorphism
\begin{align*}
\varphi:\Aut_{A+B_{\tors}}(B)\cong \Hom(B/(A+B_{\tors}),B_{\tors})
\end{align*}
which sends any $\sigma\in \Aut_{A+B_{\tors}}(B)$ to the group homomorphism $[b]\mapsto \sigma(b)-b$.
\end{prop}
\begin{proof}
Let $\sigma\in \Aut_{A+B_{\tors}}(B)$. By Lemma \ref{lem:preservesDivisionPoints} we can define a map
\begin{align*}
\varphi_{\sigma}:B/(A+B_{\tors})&\longrightarrow B_{\tors}\\
[b]&\longmapsto \sigma(b)-b
\end{align*}
which is clearly a group homomorphism.
We claim that the map
\begin{align*}
\varphi:\Aut_{A+B_{\tors}}(B)&\longrightarrow \Hom(B/(A+B_{\tors}),B_{\tors})\\
\sigma&\longmapsto\varphi_\sigma
\end{align*}
is also group homomorphism. To see this, let $\sigma,\tau\in \Aut_{A+B_{\tors}}(B)$. Notice that, since $\tau(b)-b\in B_{\tors}$ for every $b\in B$, we have $\sigma(\tau(b)-b)=\tau(b)-b$. Then we have
\begin{align*}
\varphi_{\sigma\tau}([b])&=\sigma(\tau(b))-b=\\
&=\sigma(\tau(b))-b+\tau(b)-b-\sigma(\tau(b)-b)=\\
&=\tau(b)-b+\sigma(b)-b=\\
&=\varphi_\sigma([b])+\varphi_\tau([b])
\end{align*}
which proves our claim.

The homomorphism $\varphi$ is injective, because if $\varphi_\sigma=0$ then $\sigma(b)=b$ for all $b\in B$. To see that $\varphi$ is surjective, for any $\psi\in \Hom(B/(A+B_{\tors}), B_{\tors})$ let
\begin{align*}
\sigma_\psi:B&\longrightarrow B\\
b&\longmapsto b+\psi([b])
\end{align*}
which is clearly a group homomorphism that is the identity on $A+B_{\tors}$. It is also injective, because if $b+\psi([b])=0$ then $b=-\psi([b])$ must be a torsion point, hence $-b=\psi([b])=\psi(0)=0$. By Lemma \ref{lem:InjSurj}, we have $\sigma_\psi\in\Aut_{A+B_{\tors}}(B)$ and clearly $\varphi_{\sigma_{\psi}}=\psi$, so $\varphi$ is surjective.
We conclude that $\varphi$ is an isomorphism.
\end{proof}

Combining the previous results, we obtain a fundamental exact sequence that provides our framework for the study of Kummer extensions.

\begin{prop}[{\cite[Corollary 3.12 and Corollary 3.18]{pmaster}}]
\label{prop:SequenceAut}
Let $B$ be an $s$-extension of $A$. There is an exact sequence
\begin{align*}
0\to \Hom\left(\frac{B}{A+B_{\tors}},B_{\tors}\right)\to \Aut_{A}(B)\to\Aut_{A_{\tors}}(B_{\tors})\to 1\,.
\end{align*}
Moreover, the group $\Aut_{A_{\tors}}(B_{\tors})$ acts on $\Hom\left(B/(A+B_{\tors}),B_{\tors}\right)$ by composition.
\end{prop}
\begin{proof}
Notice that $B_{\tors}$ is $A$-normal in $B$ by Example \ref{exa:criterion}, so the restriction map $\Aut_{A}(B)\to\Aut_{A_{\tors}}(B_{\tors})$ is surjective by Proposition \ref{prop:ResIsSurj}, and its kernel is $\Aut_{A+B_{\tors}}(B)$. By Proposition \ref{prop:AutIsHom} we have $\Aut_{A+B_{\tors}}(B)\cong \Hom(B/(A+B_{\tors}),B_{\tors})$, so we get the desired exact sequence.

It follows from the existence of the exact sequence above and by 
Remark \ref{rem:exactSequenceAction} that the group $\Aut_{A_{\tors}}(B_{\tors})$
acts naturally on $\Hom\left(B/(A+B_{\tors}),B_{\tors}\right)$ by conjugation.
Let now $\psi\in \Hom\left(B/(A+B_{\tors}),B_{\tors}\right)$ correspond to the automorphism $\sigma_\psi:b\to b+\psi([b])$ via the isomorphism of Proposition \ref{prop:AutIsHom}, and let $\tau \in \Aut_{A_{\tors}}(B_{\tors})$. Let moreover $\tilde\tau$ be any lift of $\tau$ to $\Aut_A(B)$. Then for every $b\in B$ we have
\begin{align*}
(\tilde{\tau}\circ\sigma_\psi\circ\tilde \tau^{-1})(b)&=\tilde\tau\left(\tilde\tau^{-1}(b)+\psi([\tilde\tau^{-1}(b)])\right)=\\
&= b+\tilde\tau\left(\psi([\tilde\tau^{-1}(b)])\right)
\end{align*}
and since $\tilde\tau^{-1}$ fixes $A$, as in the proof of Proposition \ref{prop:AutIsHom} we have that $\tilde\tau^{-1}(b)-b\in B_{\tors}$. It follows that $\psi([\tilde\tau^{-1}(b)])=\psi([b])$, so
\begin{align*}
(\tilde{\tau}\circ\sigma_\psi\circ\tilde \tau^{-1})(b)=b+\tilde\tau(\psi([b]))=b+(\tau\circ\psi)([b]),
\end{align*}
where the last equality follows from the fact that $\psi([b])\in B_{\tors}$. We conclude that the natural action of $\Aut_{A_{\tors}}(B_{\tors})$ on $\Hom\left(B/(A+B_{\tors}),B_{\tors}\right)$ is given by composition.
\end{proof}

\subsection{Profinite structure of automorphism groups}
\label{sec:profinite}

Fix for the remainder of this section a finitely generated abelian group $A$. For any $s$-extension $B$ of $A$ and for any positive integer $n$ we can consider the group $B_n$ and its automorphism group $\Aut_A(B_n)$ which, according to the following proposition, is finite.

\begin{prop}
\label{prop:FiniteAutIsFinite}
Let $B$ be an $s$-extension of $A$ and assume that $B/A$ has finite exponent. Then the automorphism group $\Aut_A(B)$ is finite.
\end{prop}
\begin{proof}
In view of Proposition \ref{prop:SequenceAut} it is enough to prove that
$\Hom\left({B}/(A+B_{\tors}),B_{\tors}\right)$ and $\Aut_{A_{\tors}}(B_{\tors})$ are finite. But this follows from the fact that both $B_{\tors}$ and $B/(A+B_{\tors})$ are finite, since $A$ is finitely generated, $B/A$ has finite exponent and $B_{\tors}$ embeds in $(\Q/\Z)^s$.
\end{proof}

Let $B$ be an $s$-extension of $A$. By Proposition \ref{prop:SequenceAut} for every positive $n$ we have an exact sequence
\begin{align*}
0\to \Hom\left(\frac{B_n}{A+B_{n,\tors}},B_{n,\tors}\right)\to \Aut_{A}(B_n)\to\Aut_{A_{\tors}}(B_{n,\tors})\to 1
\end{align*}

and for every $n\mid m$ the restriction maps make the following diagram commute:
\begin{equation*}
  \begin{tikzcd}[column sep=1.9em] 
0\arrow[r]&\Hom\left(\dfrac{B_m}{A+B_{m,\tors}},B_{m,\tors}\right) \arrow[r] \arrow[d]&\Aut_{A}(B_m)\arrow[r] \arrow[d]& \Aut_{A_{\tors}}(B_{m,\tors})\arrow[d]\arrow[r]&1\\
0\arrow[r]&\Hom\left(\dfrac{B_n}{A+B_{n,\tors}},B_{n,\tors}\right) \arrow[r] &\Aut_{A}(B_n)\arrow[r] & \Aut_{A_{\tors}}(B_{n,\tors})\arrow[r]&1
\end{tikzcd}
\end{equation*}
Notice that the rows of this diagram are exact and that every vertical map is surjective by Propostion \ref{prop:ResIsSurj}. In fact, we have
\begin{itemize}
  \item The map on the left is, once we apply Proposition \ref{prop:AutIsHom},
    the restriction map
    \begin{align*}
      \Aut_{A+B_{m,\tors}}(B_m)\to \Aut_{A+B_{n,\tors}}(B_n)
    \end{align*}
    and $A+B_{n,\tors}$ is $A$-normal in $A+B_{m,\tors}$ by Proposition \ref{prop:NormalityCriterion} (notice that the image of any $A$-homomorphism from $A+B_{n,\tors}$ to $A+B_{m,\tors}$ is contained in  $A+B_{n,\tors}$).
\item The group $B_n$ is $A$-normal in $B_m$ by Example \ref{exa:criterion}(2) and Proposition \ref{prop:NormalityCriterion}.
\item The groups $B_{n,\tors}$ and $B_{m,\tors}$ are $s$-extensions of $A_{\tors}$, and $B_{n,\tors}$ is $A_{\tors}$-normal in $B_{m,\tors}$ by Example \ref{exa:criterion}(1) and Proposition \ref{prop:NormalityCriterion}.
\end{itemize}

\begin{prop}
\label{prop:ProjLimit}
Let $B$ be an $s$-extension of $A$.
The groups $\Aut_A(B_n)$ together with the natural restriction maps $\rho_{nm}:\Aut_{A}(B_m)\to \Aut_{A}(B_n)$ for $n\mid m$ form a projective system. The group $\Aut_A(B)$ together with the natural restriction maps $\rho_n:\Aut_A(B)\to \Aut_{A}(B_n)$ is the limit of this projective system.
\end{prop}
\begin{proof}
By Proposition \ref{prop:ResIsSurj} the restriction map $\rho_m:\Aut_A(B)\to \Aut_{A}(B_m)$ is surjective for every $m$. Since for every $n\mid m$ we have $\rho_n=\rho_{nm}\circ\rho_m$, the map $\rho_{nm}$ is surjective as well. These maps are clearly compatible, so they form a projective system.

Let $G$ be any group with a compatible system of maps $\varphi_n:G\to\Aut_{A}(B_n)$. Then we can define a map $\varphi:G\to \Aut_A(B)$ by letting for every $g\in G$ and every $b\in B$
\begin{align*}
\varphi(g)(b):=\varphi_n(g)(b)
\end{align*}
where $n$ is such that $b\in B_n$. It is easy to check that this map is well-defined and that it is the unique map $G\to \Aut_A(B)$ compatible with the projections.
\end{proof}

From the above proposition it follows that the projective limit of these exact sequences is the same exact sequence of Proposition \ref{prop:SequenceAut}:
\begin{align*}
0\to \Hom\left(\frac{B}{A+B_{\tors}},B_{\tors}\right)\to \Aut_{A}(B)\to\Aut_{A_{\tors}}(B_{\tors})\to 1\,.
\end{align*}

Since this sequence is a projective limit we can endow the groups involved with the natural profinite topology by giving each finite group the discrete topology. The maps appearing in the exact sequence above are then continuous and, in particular, $\Hom\left(B/(A+B_{\tors}),B_{\tors}\right)$ and $\Aut_{A_{\tors}}(B_{\tors})$ have the subspace and quotient topology, respectively (see Lemma \ref{lem:generalTopologyExactSequenceLimit}). Notice also that $\Hom\left(B/(A+B_{\tors}),B_{\tors}\right)$, being the kernel of a continuous homomorphism, is a closed normal subgroup of $\Aut_A(B)$.

We have obtained the following refinement of Proposition \ref{prop:SequenceAut}.

\begin{prop}
\label{prop:SequenceAutContinuous}
Let $B$ be an $s$-extension of $A$.
The group $\Aut_A(B)$ together with the natural restriction maps is the projective limit of the finite groups $\Aut_{A}(B_n)$, thus it is a profinite group.
In particular, $\Aut_A(B)$ is a compact Hausdorff topological group.

There is an exact sequence of profinite groups
\begin{align*}
0\to \Hom\left(\frac{B}{A+B_{\tors}},B_{\tors}\right)\to \Aut_{A}(B)\to\Aut_{A_{\tors}}(B_{\tors})\to 1\,.
\end{align*}
Moreover, the group $\Aut_{A_{\tors}}(B_{\tors})$ acts continuously on $\Hom\left(B/(A+B_{\tors}),B_{\tors}\right)$ by composition.
\end{prop}

\subsection{Full $s$-extensions}

In this section we give a characterization of the \emph{maximal $s$-extensions} of \cite[Section 2.2]{pmaster}. We will not prove here the maximality of these extensions in the sense of \cite[Theorem 2.6]{pmaster}, hence the change of name to \emph{full $s$-extensions.} Our motivation for the study of these kind of extensions is that they provide a useful abstraction for the set of points of a commutative algebraic group that have a multiple in a fixed subgroup of rational points, in other words it is ``full'' of all division points. However, the equivalence of the two definitions follows immediately from Proposition \ref{prop:StructureOfB}.

\begin{defi}
Let $A$ be a finitely generated abelian group. An $s$-extension $\Gamma$ of $A$ is called \emph{full} if $\Gamma$ is a divisible abelian group and $\Gamma_{\tors}\cong (\Q/\Z)^s$.
\end{defi}

\begin{rem}
Recall from Remark \ref{rem:necCondAdmitsSExtension} that a necessary condition
for $A$ to admit any $s$-extension is that $A_{\tors}$ can be embedded in
$(\Q/\Z)^s$. This condition is also sufficient for $A$ to admit a full
$s$-extension. To see this, fix an isomorphism $A\cong \Z^{\rk(A)}\oplus T$,
where $T$ is a finite subgroup of $(\Q/\Z)^s$. Then the natural inclusion
$\Z^{\rk(A)}\oplus T\hookrightarrow \Q^{\rk(A)}\oplus (\Q/\Z)^s$ realizes
$\Q^{\rk(A)}\oplus (\Q/\Z)^s$ as a full $s$-extension of $A$.
\end{rem}

\begin{rem}
\label{rem:DirectSum}
Let $\Gamma$ be a full $s$-extension of a finitely generated abelian group $A$. Then $\Gamma_{\tors}\cong (\Q/\Z)^s$ is a divisible abelian group. It follows that the exact sequence
\begin{align*}
0\to \Gamma_{\tors}\to \Gamma\to\Gamma/\Gamma_{\tors}\to 0
\end{align*}
splits (non-canonically), so that $\Gamma\cong (\Gamma/\Gamma_{\tors})\oplus\Gamma_{\tors}\cong (\Gamma/\Gamma_{\tors})\oplus (\Q/\Z)^s$.
\end{rem}

The following proposition shows in particular that a finitely generated abelian group $A$ can have at most one full $s$-extension, up to (a not necessarily unique) isomorphism.

\begin{prop}
\label{prop:StructureOfB}
Let $A$ be a finitely generated abelian group of rank $r>0$ which admits a full $s$-extension $\Gamma$. There is a canonical isomorphism
\begin{align}
\label{eq:CanonicalIsomStructureofB}
\Gamma/\Gamma_{\tors}\overset{\sim}{\to} A\otimes_{\Z}\Q
\end{align}
that sends the subgroup $A/A_{\tors}$ of $\Gamma/\Gamma_{\tors}$ to $\Abar:=\set{a\otimes 1\mid a\in A}$.

Moreover, there is an isomorphism
\begin{align}
\label{eq:nonCanonicalIsomStructureofB}
\Gamma\overset{\sim}{\to} \Q^r\oplus (\Q/\Z)^s
\end{align}
that sends $A$ to $\Z^r\subseteq \Q^r$.
\end{prop}

\begin{proof}
Since $\Gamma/A$ is torsion, for every $b\in \Gamma$ there is an integer $n\geq 1$
such that $nb\in A$. Let $n_b:=\min\{n\in \NN_{\geq1}\mid nb\in A\}$. We define a map
\begin{align*}
\psi:\Gamma&\longrightarrow A\otimes_\Z\Q\\
b&\longmapsto (n_bb)\otimes \frac{1}{n_b}.
\end{align*}
The map $\psi$ is a group homomorphism. To see this, notice first that for every $b\in \Gamma$ and every $n\in\mathbb{N}_{\geq 1}$ such that $nb\in A$ we have $(nb)\otimes\frac{1}{n}=(n_bb)\otimes\frac{1}{n_b}$. Then for every $b,c\in \Gamma$ we have
\begin{align*}
\psi(b+c)=&n_{b+c}(b+c)\otimes\frac{1}{n_{b+c}}=n_bn_c(b+c)\otimes\frac{1}{n_bn_c}=\\=&(n_bn_cb)\otimes\frac{1}{n_bn_c}+(n_bn_cc)\otimes\frac{1}{n_bn_c}=\\=&(n_bb)\otimes\frac{1}{n_b}+(n_cc)\otimes\frac{1}{n_c}=\\
=&\psi(b)+\psi(c).
\end{align*}
The map $\psi$ is also surjective: in fact, let $a\in A$ and $n\in\NN_{\geq 1}$. Since $\Gamma$ is divisible, there must be an element $b\in \Gamma$ such that $nb=a$, and thus $\psi(b)=a\otimes\frac{1}{n}$.

Now we show that the $\ker\psi=\Gamma_{\tors}$. If $b\in \Gamma$ has order $n\geq 1$, then $\psi(b)=(nb)\otimes\frac{1}{n}=0$, showing that $b\in \ker\psi$. On the other hand, if $\psi(b)=(n_bb)\otimes \frac{1}{n_b}=0$, then necessarily $n_bb=0$, so that $b\in \Gamma_{\tors}$. So we get an isomorphism which sends $A/A_{\tors}$ to $\Abar$.

For the second part, since $A$ has rank $r$ we have $A\otimes_{\Z}\Q\cong \Q^r$. It follows from the first part that there is an isomorphism $\Gamma/\Gamma_{\tors}\overset{\sim}{\to}\Q^r$ that sends $A/A_{\tors}$ to $\Z^r\subseteq \Q^r$.
The conclusion follows by combining this with any isomorphism $\Gamma\overset{\sim}{\to}(\Gamma/\Gamma_{\tors})\oplus (\Q/\Z)^s$ (see Remark \ref{rem:DirectSum}).
\end{proof}

\begin{rem}
In Proposition \ref{prop:StructureOfB} the isomorphism \eqref{eq:CanonicalIsomStructureofB} is canonical, while the isomorphism \eqref{eq:nonCanonicalIsomStructureofB} depends on the choice of three isomorphisms: an isomorphism between $A\otimes_{\Z}\Q$ and $\Q^r$ (or, equivalently, a choice of a $\Z$-basis of $A/A_{\tors}$), a splitting isomorphism $\Gamma\cong (\Gamma/\Gamma_{\tors})\oplus \Gamma_{\tors}$ (see Remark \ref{rem:DirectSum}) and an isomorphism $\Gamma_{\tors}\cong (\Q/\Z)^s$.
\end{rem}

\subsection{Automorphisms of full $s$-extensions of torsion-free groups}

For this section, let $A$ be a finitely generated and torsion-free abelian group of rank $r>0$ and let $\Gamma$ be a full $s$-extension of $A$.
Notice that, since $A_{\tors}=0$, we have $\Aut_{A_{\tors}}(\Gamma_{\tors})=\Aut(\Gamma_{\tors})$ and $\Gamma_{n,\tors}=\Gamma[n]$ for every $n>0$. By Proposition \ref{prop:StructureOfB} we can fix an isomorphism
\begin{align*}
\Phi:\Gamma\xrightarrow{\sim} \Q^r\oplus (\Q/\Z)^s
\end{align*}
that maps $A$ onto $\Z^r\subseteq \Q^r$. This induces isomorphisms
\begin{align*}
\Phi_{\kumm}:\dfrac{\Gamma}{A+\Gamma_{\tors}}\xrightarrow{\sim} (\Q/\Z)^r,\qquad \Phi_{\tors}:\Gamma_{\tors}\xrightarrow{\sim} (\Q/\Z)^s.
\end{align*}
Recall from Remark \ref{rem:HigherDimPontryagin} that we have canonical isomorphisms
\begin{align*}
\Aut((\Q/\Z)^s)\cong \GL_s(\hat\Z), \qquad \Hom((\Q/\Z)^r,(\Q/\Z)^s)\cong \Mat_{s\times r}(\hat\Z)
\end{align*}
under which the action of $\Aut((\Q/\Z)^s)$ on $\Hom((\Q/\Z)^r,(\Q/\Z)^s)$ given by composition becomes matrix multiplication on the left. So we get isomorphisms
\begin{align*}
\Phi^*_{\kumm}:\Hom\left(\dfrac{\Gamma}{A+\Gamma_{\tors}},\Gamma_{\tors}\right)\xrightarrow{\sim} \Mat_{s\times r}(\hat\Z), \qquad 
\Phi_{\tors}^*:\Aut(\Gamma_{\tors})\xrightarrow{\sim} \GL_s(\hat\Z).
\end{align*}

On the finite level, these isomorphisms induce, for every $n>0$, isomorphisms
\begin{align*}
\psi_n: \Hom\left(\dfrac{\Gamma_n}{A+\Gamma[n]},\Gamma[n]\right)\xrightarrow{\sim} \Mat_{s\times r}\left({\Z}/{n\Z}\right), \qquad\varphi_n:\Aut(\Gamma[n])\xrightarrow{\sim}  \GL_s\left({\Z}/{n\Z}\right)
\end{align*}
which are compatible with the natural projections, in the sense that for every $n\mid m$ the diagrams
\begin{equation*}
\begin{tikzcd}[column sep=1.9em] 
\Hom\left(\dfrac{\Gamma_m}{A+\Gamma[m]},\Gamma[m]\right)\arrow{r}{\psi_m} \arrow{d} &\Mat_{s\times r}\left({\Z}/{m\Z}\right)\arrow{d}\\
\Hom\left(\dfrac{\Gamma_n}{A+\Gamma[n]},\Gamma[n]\right)\arrow{r}{\psi_n}  &\Mat_{s\times r}\left({\Z}/{n\Z}\right)
\end{tikzcd}
\end{equation*}
and
\begin{equation*}
\begin{tikzcd}[column sep=1.9em] 
\Aut(\Gamma[m])\arrow{r}{\varphi_m} \arrow{d} &\GL_s(\Z/m\Z)\arrow{d}\\
\Aut(\Gamma[n])\arrow{r}{\varphi_n}  &\GL_s(\Z/n\Z)
\end{tikzcd}
\end{equation*}
commute. This shows that the topology with which we endowed our automorphism groups coincides with the natural topology of the $\hat\Z$-matrix rings, as stated in the following proposition.

\begin{prop}
\label{prop:AutFullStruct}
Let $A$ be a finitely generated and torsion free abelian group of rank $r>0$ and let $\Gamma$ be a full $s$-extension of $A$. Consider the group $\Aut_A(\Gamma)$ with the profinite topology described in Section \ref{sec:profinite} and the groups $\Mat_{s\times r}(\hat\Z)$ and $\GL_s(\hat\Z)$ with the topology induced by the profinite topology of $\hat\Z$.

Then every isomorphism of abelian groups
\begin{align*}
\Phi:\Gamma\xrightarrow{\sim} \Q^r\oplus (\Q/\Z)^s
\end{align*}
that maps $A$ onto $\Z^r\subseteq \Q^r$ induces isomorphisms of topological groups
\begin{align*}
\Phi^*_{\kumm}:\Hom\left(\dfrac{\Gamma}{A+\Gamma_{\tors}},\Gamma_{\tors}\right)\xrightarrow{\sim} \Mat_{s\times r}(\hat\Z), \qquad \Phi_{\tors}^*:\Aut(\Gamma_{\tors})\xrightarrow{\sim} \GL_s(\hat\Z)\,.
\end{align*}
Moreover, the action of $\Aut(\Gamma_{\tors})$ on $\Hom\left({\Gamma}/(A+\Gamma_{\tors}),\Gamma_{\tors}\right)$ given by composition is identified under these isomorphisms with matrix multiplication on the left.
\end{prop}

\section{Some linear algebra}
\label{sec:LinearAlgebra}

Motivated by the results of the previous sections we will now establish some results of linear algebra over the ring $\hat\Z$. In particular, we are interested in certain properties of $\Mat_{s\times r}(\hat\Z)$ as a left $\Mat_{s\times s}(\hat\Z)$-module.

Fix for this section two non-negative integers $s$ and $r$.

\begin{prop}
\label{prop:HigherDimensionalPontryagin}
Let $R:=\Mat_{s\times s}(\hat\Z)$ and view $M:=\Mat_{s\times r}(\hat\Z)$ as a left $R$-module. Let $V\subseteq M$ be a left $R$-submodule. Assume that there is a positive integer $n$ such that, viewing the elements of $V$ as maps $(\Q/\Z)^r\to(\Q/\Z)^s$, we have
\begin{align}
\label{eqn:KernelCondition}
\bigcap_{f\in V}\ker f\subseteq \left(\Q/\Z\right)^r[n].
\end{align}
Then $V\supseteq nM$.
\end{prop}
\begin{proof}
Let $L$ denote the right $R$-module $\hat\Z^s$ of row vectors and let $N$ denote the left $R$-module $\hat\Z^s$ of column vectors. Notice that there is a natural $R$-module isomorphism
\begin{align*}
\begin{array}{cccc}
& N\otimes_{\hat\Z}L\otimes_{R} M&\to &M\\
&x\otimes y\otimes m&\mapsto& x\cdot y\cdot m
\end{array}
\end{align*}
whose inverse is
\begin{align*}
\begin{array}{cccc}
\psi: & M & \to & N\otimes_{\hat\Z}L\otimes_{R} M\\
& m& \mapsto & \sum_{i=1}^s e_i\otimes f_i\otimes m
\end{array}
\end{align*}
where $\{e_i\}$ and $\{f_i\}$ are the canonical bases for $N$ and $L$ respectively.

Consider now the abelian group $M_L:=L\otimes_{R}M$, which is isomorphic to $\hat\Z^r$ via
\begin{align*}
\begin{array}{cccc}
& L\otimes_R M & \to & \hat\Z^r\\
& y\otimes v & \mapsto & y\cdot v
\end{array}
\end{align*}
 and its subgroup
 \[V_L=\langle y\otimes v\mid y\in L,\,v\in V\rangle.\]
 Condition \eqref{eqn:KernelCondition} implies that, seeing the elements of $V_L$ as maps $(\Q/\Z)^r\to \Q/\Z$, we have $\bigcap_{f\in V_L}\ker f\subseteq(\Q/\Z)^r[n]$. Then by Pontryagin duality (Theorem \ref{thm:Pontryagin}) we have $V_L\supseteq nM_L$.

The image of $V$ in $N\otimes_{\hat\Z}L\otimes_RM$ under the isomorphism $\psi$ is
\begin{align*}
\psi(V)=\langle x\otimes y\otimes v\mid x\in N,\,y\in L,\,v\in V\rangle =
\langle x\otimes v_L\mid x\in N,\,v_L\in V_L\rangle
\end{align*}
and since
\begin{align*}
n(N\otimes_{\hat\Z}L\otimes_RM)&=\langle n(x\otimes y\otimes v)\mid x\in N,\,y\in L,\,v\in M\rangle=\\
&= \langle x\otimes n( y\otimes v)\mid x\in N,\,y\in L,\,v\in M\rangle=\\
&= \langle x\otimes w \mid x\in N,\, w\in nM_L\rangle
\end{align*}
we have
\begin{align*}
\psi(V)\supseteq n(N\otimes_{\hat\Z}L\otimes_RM)
\end{align*}
which is equivalent to $V\supseteq nM$.
\end{proof}

\begin{lem}
\label{lem:topmodules}
Let $R$ be a compact topological ring and let $M$ be a compact topological
$R$-module. Let $T\subseteq R$ be a subring of $R$ and let $S$ denote the
smallest closed subring of $R$ containing $T$. If $V\subseteq M$ is a closed
$T$-submodule, then $V$ is also an $S$-module.
\end{lem}
\begin{proof}
Let $v\in V$ and consider the continuous map
\begin{align*}
f_v:R&\to M\\
x&\mapsto xv
\end{align*}
Since $S$ is the closure of $T$ in $R$, we have
\begin{align*}
f_v(S) = f_v\left(\bigcap\set{C \mid C\text{ closed, } T\subseteq C\subseteq R}\right)\subseteq \bigcap\set{f_v(C)\mid C\text{ closed, } T\subseteq C\subseteq R}.
\end{align*}
For any closed subset $D$ of $M$ containing $f(T)$ we have
that $f^{-1}(D)$ is closed and contains $T$
and
$f(f^{-1}(D))\subseteq D$,
so $f_v(S)$ is contained in the closure of $f(T)$.

Since $V$ is a $T$-module, we have $f_v(T)\subseteq V$, and since $V$ is closed we have $f_v(S)\subseteq V$ by what we have just said. Since this holds for any $v\in V$, we conclude that $V$ is an $S$-module.
\end{proof}

The following Proposition is essentially a generalization of \cite[Proposition 4.12(1)]{lt}. 
\begin{prop}
\label{prop:MoveVectorsAround}
Let $R:=\Mat_{s\times s}(\Z_\ell)$ and view $M:=\Mat_{s\times r}(\Z_\ell)$ as a left $R$-module. Let $H$ be a closed subgroup of $\GL_s(\Z_\ell)$ and $V\subseteq M$ a closed left $H$-submodule. Let $W=R\cdot V$ and let $S$ denote the closed $\Z_\ell$-subalgebra of $R$ generated by $H$. Suppose that there are non-negative integers $n$ and $m$ such that
\begin{enumerate}[(1)]
\item $W\supseteq \ell^n M$ and
\item $S\supseteq\ell^mR$.
\end{enumerate}
Then we have $V\supseteq \ell^{n+m}M$.
\end{prop}
\begin{proof}
Let $T$ denote the (not necessarily closed) $\Z_\ell$-subalgebra of $R$ generated by $H$, so that $S$ is the closure of $T$. It is clear that $V$, being both a $\Z_\ell$-module and an $H$-module, is a $T$-module. Since it is closed, $V$ is also an $S$-module by Lemma \ref{lem:topmodules} above.

Then we have $V\supseteq S\cdot V\supseteq\ell^mR\cdot V=\ell^m W\supseteq \ell^m\cdot\ell^nM$.
\end{proof}

The following result is an adelic version of Proposition \ref{prop:MoveVectorsAround}.

\begin{prop}
\label{prop:AdelicMoveVectors}
Let $R:=\Mat_{s\times s}(\hat\Z)$ and view $M:=\Mat_{s\times r}(\hat\Z)$ as a left $R$-module. Let $H$ be a closed subgroup of $\GL_s(\hat\Z)$ and let $V\subseteq M$ be a closed left $H$-submodule. Let $W=R\cdot V$ and, for every prime $\ell$, let $H_\ell$ denote the image of $H$ under the projection $\GL_s(\hat\Z)\to\GL_s(\Z_\ell)$ and let $\Z_\ell[H_\ell]$ denote the closed sub-$\Z_\ell$-algebra of $\Mat_{s\times s}(\Z_\ell)$ generated by $H_\ell$. Suppose that there are positive integers $n$ and $m$ such that
\begin{enumerate}[(1)]
\item $W\supseteq n M$;
\item For every prime $\ell$ we have $\Z_\ell[H_\ell]\supseteq m \Mat_{s\times s}(\Z_\ell)$.
\end{enumerate}
Then we have $V\supseteq nmM$.
\end{prop}
\begin{proof}
Let $R_\ell:=\Mat_{s\times s}(\Z_\ell)$ and $M_\ell:=\Mat_{s\times r}(\Z_\ell)$, so that
\begin{align*}
R=\prod_\ell R_\ell \qquad \text{and} \qquad M=\prod_\ell M_\ell.
\end{align*}
Let moreover $V_\ell$ and $W_\ell$ denote the images of $V$ and $W$ in $M_\ell$, respectively. Notice that $V_\ell$ is an $H_\ell$-submodule of $M_\ell$ and that $W_\ell$ is the $R_\ell$-submodule of $M_\ell$ generated by $V_\ell$.

By (1) we have that $W_\ell$ contains the image of $nM$ in $M_\ell$, which is $nM_\ell$. By (2) we have $\Z_\ell[H_\ell]\supseteq m\Mat_{s\times s}(\Z_\ell)$, so we can apply Proposition \ref{prop:MoveVectorsAround} and deduce that $V_\ell\supseteq nmM_\ell$.

We claim that $V=\prod_\ell V_\ell$, seen as a subgroup of $\prod_\ell M_\ell$. Clearly $V\subseteq \prod_\ell V_\ell$, since every $v\in V$ is equal to the tuple $(e_\ell v)_\ell$, where $e_\ell\in \hat\Z=\prod \Z_p$ is the element whose $\ell$-component is $1$ and whose $p$-component is $0$ for all $p\neq \ell$. For the other inclusion, let $(w_\ell)_\ell\in \prod_\ell V_\ell$. Since $V_\ell$ is the image of $V$ under the natural projection, for every $\ell$ there must be $\tilde w_\ell\in V$ whose $\ell$-component is $w_\ell$. Then the infinite sum
\begin{align*}
\sum_\ell e_\ell\tilde w_\ell
\end{align*}
converges to $(w_\ell)_\ell$ in $M$: consider the sequence of partial sums
\begin{align*}
\{x_k\}_{k\in \NN}=\left\{ \sum_{\ell\leq k}e_\ell\tilde w_\ell \right\}_{k\in \NN}
\end{align*}
and let $U\subseteq M$ be an open neighbourhood of $(w_\ell)_\ell$, which must be of the form
\begin{align*}
\prod_{\ell\leq N}U_\ell\times \prod_{\ell>N}M_\ell
\end{align*}
for some integer $N$ and some open neighbourhoods $U_\ell$ of $w_\ell$ in $M_\ell$; then clearly $x_k\in U$ for all $k\geq N$.

Since $V$ is closed in $M$, we must then have $(w_\ell)_\ell\in V$, which shows that $V=\prod_\ell V_\ell$.

Since for every prime $\ell$ the multiplication-by-$\ell$ endomorphism on a $\hat\Z$-module is invertible on all prime-to-$\ell$ components, we have $\prod_\ell nmM_\ell=\prod_\ell\ell^{v_\ell(nm)}M_\ell=nm M$, so
\begin{align*}
V=\prod_\ell V_\ell \supseteq \prod_{\ell}nmM_\ell=nmM
\end{align*}
and we conclude.
\end{proof}
\section{General entanglement theory}
\label{sec:GeneralTheory}

\subsection{Initial remarks and definitions}

Fix a number field $K$ and an algebraic closure $\Kbar$ of $K$. Let $G$ be a commutative connected algebraic group over $K$. It is well-known that there is a non-negative integer $s$, depending only on $G$, such that $G(\Kbar)[n]\cong (\Z/n\Z)^s$ for all integers $n> 1$. For example, if $G$ is an abelian variety of dimension $g$, we have $s=2g$.

Let $A\subseteq G(K)$ be a finitely generated and torsion-free subgroup of rank $r>0$ and consider the \emph{divisible hull} of $A$ in $G(\Kbar)$
\begin{align}
\label{eq:radicalGroup}
\Gamma:=\set{P\in G(\Kbar)\mid \exists\,n\in\NN_{\geq 1}:\,nP\in A}
\end{align}

which is a subgroup of $G(\Kbar)$ and a full $s$-extension of $A$.

We have $\Gamma_{\tors}=G(\Kbar)_{\tors}$, which we will also denote by $G_{\tors}$.
We also have
\begin{align*}
A+G(K)_{\tors}\subseteq\Gamma\cap G(K).
\end{align*}
The quotient group $(\Gamma\cap G(K))/ (A+G(K)_{\tors})$, being a quotient of a subgroup of $\Gamma/A$, is always a torsion group.

\begin{defi}
\label{def:divisibility}
We call any integer $d_A>1$ such that $d_A(\Gamma\cap G(K))\subseteq A+G(K)_{\tors}$ a \emph{divisibility parameter} for $A$ in $G(K)$. If such an integer exists, we say that \emph{$A$ has finite divisibility} in $G(K)$.
\end{defi}

\begin{exa}
\label{exa:divisibilityParam}
\begin{enumerate}[(1)]
\item If $G(K)$ is finitely generated, every torsion-free subgroup $A\subseteq G(K)$ has finite divisibility in $G(K)$: in fact, the abelian group $(\Gamma\cap G(K))/ (A+G(K)_{\tors})$ is torsion and finitely generated, so it is finite.
\item Let $G=\mathbb{G}_m$ be the multiplicative group, so that $s=1$. In this case $G(K)=K^\times$ is not finitely generated, but it still holds that every finitely generated $A\subseteq G(K)$ has finite divisibility. In order to prove this it is enough to show that for every prime number $\ell$ there is a non-negative integer $m_\ell$ such that the $\ell$-power torsion of $(\Gamma\cap G(K))/ (A+G(K)_{\tors})$ is contained in
\begin{align*}
\frac{\Gamma\cap G(K)}{A+G(K)_{\tors}}[\ell^{m_\ell}]
\end{align*}
and that we can take $m_\ell=0$ for all but finitely many primes $\ell$. The first part is just \cite[Lemma 12]{debryperucca}. As for the second part, assume that $A$ admits a strongly $\ell$-independent basis $a_1,\dots, a_r$ as in \cite[Definition 2.1]{peruccasgobba}, which is true for all but finitely many $\ell$ by \cite[Theorem 2.7]{peruccasgobba}. Let $b\in \Gamma\cap K^{\times}$ be such that $b^{\ell^m}\in A\cdot \mu(K)$ for some $m\geq 1$. Then
\begin{align*}
b^{\ell^m}=\zeta\cdot \prod_{i=1}^ra_i^{x_i}
\end{align*}
for some $x_1,\dots,x_r\in \Z$ and some root of unity $\zeta\in K$ of order a power of $\ell$. Since the $a_i$ are strongly $\ell$-independent, every $x_i$ is divisible by $\ell^m$. This means that $b\in A\cdot \mu(K)=A+G(K)_{\tors}$, so we can take $m_\ell=0$.

Notice that the cited results are fully explicit, so a divisibility parameter for $A$ is effectively computable.
\item Let $G=\G_a$ be the additive group, so that $s=0$. In this case no subgroup $A\subseteq G(K)$ has finite divisibility. In fact we have
\begin{align*}
\Gamma=\set{b\in \overline K\mid \exists\, n\in\NN_{\geq 1}\text{ such that }nb\in A}\subseteq K.
\end{align*}
Then $(\Gamma\cap G(K))/A=\Gamma/A$ contains elements of unbounded order. Since $\Gamma\subseteq G(K)$, \emph{Kummer theory for the additive group is trivial}.
\end{enumerate}
\end{exa}

\subsection{Torsion and Kummer representations and the entanglement group}
Fix for the rest of the section a finitely generated subgroup $A\subseteq G(K)$.
For simplicity, we will denote $K(G_{\tors})$ by $K_\infty$. We are interested in studying the tower of extensions $K(\Gamma)\mid K_\infty\mid K$. Notice that $K(\Gamma)$ is a Galois extension of $K$: in fact it is the union of its finite subextensions of the form $K(\Gamma_n)$, where $\Gamma_n=\{P\in G(\Kbar)\mid nP\in A\}$, which are Galois. Similarly, $K_\infty\mid K$ is Galois, since it is the union of the finite Galois extensions $K_n:=K(G[n])$ of $K$.

The action of $\Gal(\Kbar\mid K)$ on $G(\Kbar)$ gives rise, for every $n\geq1$, to injective homomorphisms
\begin{align*}
  \Gal(K(\Gamma_n)\mid K_n) &\hookrightarrow \Aut_{A+G[n]}(\Gamma_n)\cong\Hom\left(\frac{\Gamma_n}{A+G[n]},G[n]\right),\\[0.6em]
  \Gal(K(\Gamma_n)\mid K) &\hookrightarrow \Aut_A(\Gamma_n),\\[1em]
  \Gal(K_n\mid K) &\hookrightarrow \Aut(G[n])
\end{align*}
which by Proposition \ref{prop:SequenceAutContinuous} fit into the commutative diagram with exact rows
\begin{equation*}
\begin{tikzcd}[column sep=1.9em] 
1\arrow[r]&\Gal(K(\Gamma_n)\mid K_n) \arrow[r] \arrow[hook,d]&\Gal(K(\Gamma_n)\mid K)\arrow[r] \arrow[hook,d]& \Gal(K_n\mid K)\arrow[hook,d]\arrow[r]&1\\
0\arrow[r]&\Hom\left(\dfrac{\Gamma_n}{A+G[n]},G[n]\right) \arrow[r] &\Aut_A(\Gamma_n)\arrow[r] & \Aut (G[n])\arrow[r]&1
\end{tikzcd}
\end{equation*}

Taking the projective limit we obtain the following commutative diagram
of topological groups with exact rows:

\begin{equation*}
\begin{tikzcd}[column sep=1.9em] 
1\arrow[r]&\Gal(K(\Gamma)\mid K_{\infty}) \arrow[r] \arrow[hook,d]&\Gal(K(\Gamma)\mid K)\arrow[r] \arrow[hook,d]& \Gal(K_\infty\mid K)\arrow[hook,d]\arrow[r]&1\\
0\arrow[r]&\Hom\left(\dfrac{\Gamma}{A+G_{\tors}},G_{\tors}\right) \arrow[r] &\Aut_A(\Gamma)\arrow[r] & \Aut (G_{\tors})\arrow[r]&1
\end{tikzcd}
\end{equation*}

and the Krull topology on the Galois groups coincides with the subspace topology with respect to the automorphism groups.

\begin{defi}
We call the cokernel of the above defined map
\begin{align*}
\Gal(K(\Gamma)\mid K_\infty) \hookrightarrow \Hom\left(\dfrac{\Gamma}{A+G_{\tors}},G_{\tors}\right)
\end{align*}
the \textbf{entanglement group} of $A$, and we denote it by $\Ent(A)$.
\end{defi}

Fixing an isomorphism as in Proposition \ref{prop:StructureOfB}
\begin{align*}
\Phi:\Gamma\overset\sim\to \Q^r\oplus (\Q/\Z)^s
\end{align*}
that maps $A$ to $\Z^r\subseteq \Q^r$, we get by Proposition \ref{prop:AutFullStruct} isomorphisms of topological groups
\begin{align*}
\Phi^*_{\kumm}:\Hom\left(\dfrac{\Gamma}{A+\Gamma_{\tors}},\Gamma_{\tors}\right)\xrightarrow{\sim} \Mat_{s\times r}(\hat\Z),
\qquad
\Phi_{\tors}^*:\Aut(\Gamma_{\tors})\xrightarrow{\sim} \GL_s(\hat\Z).
\end{align*}
Then we get a diagram with exact rows
\begin{equation*}
\begin{tikzcd}[column sep=1.9em] 
1\arrow[r]&\Gal(K(\Gamma)\mid K_{\infty}) \arrow[r] \arrow[hook,d]&\Gal(K(\Gamma)\mid K)\arrow[r] \arrow[hook,d]& \Gal(K_\infty\mid K)\arrow[hook,d]\arrow[r]&1\\
0\arrow[r]&\Mat_{s\times r}(\hat\Z) \arrow[r] &\Aut_{\Z^r}\left(\Q^r\oplus (\Q/\Z)^s\right)\arrow[r] & \GL_s(\hat\Z)\arrow[r]&1
\end{tikzcd}
\end{equation*}
which we will refer to as the \textbf{torsion-Kummer representation} related to $A$. We will also call the map
\begin{align*}
\Gal(K(\Gamma)\mid K_\infty)\hookrightarrow \Mat_{s\times r}(\hat\Z)
\end{align*}
the \textbf{Kummer representation}, and the map
\begin{align*}
\Gal(K_\infty\mid K)\hookrightarrow \GL_s(\hat\Z)
\end{align*}
the \textbf{torsion representation}.

\begin{defi}
We will denote by $H(G)$ the image of $\Gal(K_\infty\mid K)$ in $\GL_s(\hat\Z)$ and by $V(A)$ the image of $\Gal(K(\Gamma)\mid K_\infty)$ in $\Mat_{s\times r}(\hat\Z)$. 
\end{defi}

Since all groups appearing in the diagram above are profinite and all the maps are continuous, it follows that $V(A)$ and $H(G)$ are closed subgroups of $\Mat_{s\times r}(\hat\Z)$ and $\GL_s(\hat\Z)$, respectively. One of our goals is proving that, under certain conditions, $V(A)$ is also open. More precisely, we want to bound the order of $\Ent(A)\cong\Mat_{s\times r}(\hat\Z)/V(A)$.

\begin{rem}
  \label{rem:kummerDegDividesN}
  It follows from the existence of the Kummer representation that for any
  $n\geq 1$ the degree $[K(n^{-1}A):K(G[n])]$ divides $n^{rs}$.
\end{rem}

\begin{rem}
The definition of entanglement group given here is different from that of \cite{palenstijn}, where the entanglement group for $G=\mathbb{G}_m$ is defined as the quotient of $\Aut_A(\Gamma)$ by the image of $\Gal(K(\Gamma)\mid K)$, which in the cases considered there is a normal subgroup (see \cite[Theorem 1.6]{palenstijn}). 
In fact, the entanglement group defined here is a subgroup of that of \cite{palenstijn}.
\end{rem}

We conclude this section by remarking the following fact.

\begin{lem}
\label{lemma:entBoundsDegrees}
Let $G$ be a commutative connected algebraic group over a number field $K$ and
let $A\subseteq G(K)$ be a finitely generated, torsion-free subgroup of $G(K)$
of rank $r>0$.
If $\Ent(A)$ is finite, for every $n\geq 1$
\begin{align*}
  \frac{n^{rs}}{\left[K\left(n^{-1}A\right):K\left(G[n]\right)\right]}
  \quad \text{divides} \quad \#\Ent(A)\,.
\end{align*}
\end{lem}
  \begin{proof}
    The image of $V(A)$ under the natural quotient map
    $\Mat_{s\times r}(\hat\Z)\to \Mat_{s\times r}(\Z/n\Z)$
    is $\Gal(K_\infty(n^{-1}A)\mid K_\infty)$, so the ratio
    \begin{align*}
      \frac{n^{rs}}
        {\left[K_\infty\left(n^{-1}A\right):K_\infty\right]}
    \end{align*}
    divides $\#\Ent(A)$. In order to conclude it suffices to notice that
    \begin{align*}
      [K(n^{-1}A): K(G[n])]&=[K(n^{-1}A):K_\infty\cap K(n^{-1}A)]\cdot
      [K_\infty\cap K(n^{-1}A):K(G[n])]=\\
        &=[K_\infty\left(n^{-1}A\right):K_\infty]\cdot
          [K_\infty\cap K(n^{-1}A):K(G[n])].
    \end{align*}
  \end{proof}

\subsection{Bounding the entanglement group}

We now give some sufficient conditions for the finiteness of the entanglement group $\Ent(A)$. In particular, we want to explicitly bound its cardinality in terms of some known quantities. This will be accomplished by applying the results of Section \ref{sec:LinearAlgebra}.

Assume for the rest of this section that $A$ has finite divisibility and that
$d_A$ is a divisibility parameter for $A$ in $G(K)$. Consider the joint kernel of the elements of $V(A)$, that is
\begin{align*}
  S(A):= \bigcap_{f\in V(A)}\ker f\subseteq (\Q/\Z)^r.
\end{align*}
where we consider elements of $\Mat_{s\times r}(\hat \Z)$ as maps $(\Q/\Z)^r\to(\Q/\Z)^s$.
The image of any $[b]\in \Gamma/(A+G_{\tors})$ in $(\Q/\Z)^r$
is in the kernel of every $f\in V(A)$ if and only if $b$ is fixed by every automorphism $\sigma\in \Gal(K(\Gamma)\mid K_\infty)$, that is if and only if $b\in G(K_\infty)$. So we have
\begin{align*}
  S(A)=\overline\Phi\left(\frac{\Gamma\cap G(K_\infty)}{A+G_{\tors}}\right).
\end{align*}
where we have denoted by $\overline \Phi$ the isomorphism $\Gamma/(A+\Gamma_{\tors})
\overset{\sim}{\to}(\Q/\Z)^r$ induced by $\Phi$.  
Let
\begin{align*}
\varphi:\Gamma\cap G(K_\infty)&\longrightarrow H^1(\Gal(K_\infty\mid K),G_{\tors})
\end{align*}
be the group homomorphism that maps an element $b\in \Gamma\cap G(K_\infty)$ to the class of the cocyle $\varphi_b:\sigma\mapsto \sigma(b)-b$. Notice that $A+G_{\tors}\subseteq \ker\varphi$, because $\Gal(K_\infty \mid K)$ acts trivially on $A$ and $\varphi_t$ is a coboundary for every $t\in G_{\tors}$. So $\varphi$ gives rise to a map
\begin{align*}
S(A)\longrightarrow  H^1(\Gal(K_\infty\mid K),G_{\tors})
\end{align*}
which we also denote by $\varphi$.

\begin{prop}
\label{prop:JointKernelInH1}
The kernel of $\varphi:S(A)\to H^1(\Gal(K_\infty\mid K),G_{\tors})$ is contained in $S(A)[d_A]$. In particular, if $ H^1(\Gal(K_\infty\mid K),G_{\tors})$ has finite exponent $n$, then the exponent of $S(A)$ divides $nd_A$.
\end{prop}
\begin{proof}
Let $b\in \Gamma\cap G(K_\infty)$ and assume that $\varphi_b$ is a coboundary. We want to show that $d_Ab\in A+G_{\tors}$. Since $\varphi_b$ is a coboundary, there is $t_0\in G_{\tors}$ such that for all $\sigma\in \Gal(K_\infty\mid K)$ we have $\sigma(b)-b=\sigma(t_0)-t_0$, hence $\sigma(b-t_0)=b-t_0$. This means that $b-t_0\in \Gamma\cap G(K)$, hence $d_Ab=d_A(b-t_0)+d_At_0$. Since $d_A$ is a divisibility parameter for $A$, we have $d_A(b-t_0)\in A+G(K)_{\tors}$, so $d_Ab\in A+G_{\tors}$. Hence $d_Ab$ is zero in $S(A)$.
\end{proof}

We can finally prove the main theorem of this section. Recall that $s$ is a
non-negative integer such that $G[n]\cong (\Z/n\Z)^s$ for every $n\geq 1$ and
that $H(G)$ denotes the image of $\Gal(K_{\infty}\mid K)$ in $\GL_s(\hat\Z)$.

\begin{thm}
\label{thm:GeneralBoundEntanglement}
Let $G$ be a commutative connected algebraic group over a number field $K$ and let $A\subseteq G(K)$ be a finitely generated and torsion-free subgroup of rank $r>0$. For every prime $\ell$, let $H_\ell(G)$ denote the image of $H(G)$ under the projection $\GL_s(\hat\Z)\to\GL_s(\Z_\ell)$ and denote by $\Z_\ell[H_\ell(G)]$  the closed sub-$\Z_\ell$-algebra of $\Mat_{s\times s}(\Z_\ell)$ generated by $H_\ell(G)$. Assume that
\begin{enumerate}[(1)]
\item The group $A$ admits a divisibility parameter $d_A$ in $G(K)$.
\item There is an integer $n\geq 1$ such that $\Z_\ell[H_\ell(G)]\supseteq n\Mat_{s\times s}(\Z_\ell)$ for every prime $\ell$.
\item There is an integer $m\geq 1$ such that the exponent of $H^1(\Gal(K_\infty\mid K),G_{\tors})$ divides $m$.
\end{enumerate}
Then $V(A)$ is open in $\Mat_{r\times s}(\hat\Z)$. More precisely, the order of $\Ent(A)$ divides $(d_Anm)^{rs}$.
\end{thm}
\begin{proof}
Let $\Gamma:=\set{P\in G(\Kbar)\mid \exists\,n\in\NN_{\geq 1}:\,nP\in A}$ and fix an isomorphism $\Gamma\overset\sim\to \Q^r\oplus (\Q/\Z)^s$ that sends $A$ to $\Z^r$ as in Proposition \ref{prop:StructureOfB}, so that we get a torsion-Kummer representation as in the previous subsection.We can then identify $H(G)$ with a subgroup of $\GL_s(\hat\Z)$ and $V(A)$ with a subgroup of $\Mat_{s\times r}(\hat\Z)$, and the natural action of $H(G)$ on $V(A)$ is indentified with the usual matrix multiplication on the left (see Proposition \ref{prop:AutFullStruct}).

Thanks to conditions (1) and (3) we can apply Proposition \ref{prop:JointKernelInH1} and deduce that
\begin{align*}
S(A)=\bigcap_{f\in V(A)}\ker f\subseteq (\Q/\Z)^r[d_Am],
\end{align*}
so that by Proposition \ref{prop:HigherDimensionalPontryagin} we have that the $\GL_s(\hat\Z)$-submodule of $\Mat_{s\times r}(\hat \Z)$ generated by $V(A)$ contains $d_Am \Mat_{s\times r}(\hat\Z)$. This property and (2) allow us to apply Proposition \ref{prop:AdelicMoveVectors} and deduce that the index of $V(A)$ in $\Mat_{s\times r}(\hat\Z)$ divides $(d_Anm)^{rs}$.
\end{proof}

\begin{rem}
Let $G=\mathbb{G}_m$ and let $A$ be a finitely generated and torsion-free subgroup of $G$ of rank $r>0$. Theorem \ref{thm:GeneralBoundEntanglement} gives us another way of proving \cite[Theorem 1.1]{peruccasgobba}, which states that there exists an integer $C\geq 1$ such that for every $n\geq 1$ the ratio
\begin{align}
\label{eqn:PS1}
\frac{n^r}{\left[K\left(\zeta_n,\sqrt[n]{A}\right):K\left(\zeta_n\right)\right]}
\end{align}
divides $C$. Indeed, the ratio \eqref{eqn:PS1} always divides $\#\Ent(A)$ (Lemma \ref{lemma:entBoundsDegrees}), and we have:
\begin{enumerate}[(1)]
\item The group $A$ has finite divisibility (see Example \ref{exa:divisibilityParam}).
\item The torsion representation $\tau:\Gal(K_\infty\mid  K)\to \GL_1(\hat\Z)=\hat\Z^\times$ coincides with the adelic cyclotomic character, whose image is open in $\hat\Z^\times$; more precisely, the index of $H(\mathbb{G}_m)$ in $\hat\Z^\times$ divides $[K:\Q]$, so that $\Z_\ell[H_\ell(\mathbb{G}_m)]\supseteq [K:\Q]\Mat_{s\times s}(\Z_\ell)$ for every prime $\ell$.
\item By (2) above $H(\mathbb{G}_m)$ contains every element of $\mathbb{Z}^\times$ that is congruent to the identity modulo $[K:\Q]$; an application of Sah's Lemma (see also the proof of Proposition \ref{prop:killH1forEC}) tells us that \[[K:\Q]H^1(\Gal(K_\infty\mid K),\mathbb{G}_{m,\tors})=0\,.\]
\end{enumerate}
So by Theorem \ref{thm:GeneralBoundEntanglement} we may take $C=(d_A\cdot [K:\Q]^2)^r$.

It is worth noticing that the methods of \cite{peruccasgobba} provide a more precise bound.
\end{rem}

\section{Elliptic curves}
\label{sec:ellipticNoCM}
For this section we fix a number field $K$ with algebraic closure $\Kbar$ and an elliptic curve $E$ over $K$ with $\End_K(E)=\Z$. Moreover, we let $A$ be a torsion-free subgroup of $E(K)$ of rank $r>0$ and let $\Gamma\subseteq E(\Kbar)$ be the subgroup defined in \eqref{eq:radicalGroup}, which is a full $2$-extension of $A$.

Our goal is to apply Theorem \ref{thm:GeneralBoundEntanglement} to get an explicit bound on the cardinality of $\Ent(A)$. In order to do so, we need to study the divisibility parameter $d_A$ and the torsion representations associated with $E/K$.

\subsection{The divisibility parameter}
\label{subsec:dParam}
If a set of generators for $A$, modulo torsion in $E(K)$, is known in terms of a $\Z$-basis for $E(K)/E(K)_{\tors}$, then we can compute $d_A$ effectively. In fact, let $\overline{E(K)}=E(K)/E(K)_{\tors}$ and let $\overline A$ be the image of $A$ in $\overline{E(K)}$. Let $\mathbf{e}_1,\dots,\mathbf{e}_\rho$ be a basis for $\overline{E(K)}$ as a free $\Z$-module and let $\mathbf{a}_1,\dots,\mathbf{a}_t$ be a set of generators for $\overline A$. Write
\begin{align*}
\mathbf{a}_i=\sum_{j=1}^\rho m_{ij}\mathbf{e}_j
\end{align*}
for some integers $m_{ij}$, and let $M$ be the $\rho\times t$ matrix $(m_{ji})$ whose columns are the coordinate vectors representing the $\mathbf{a}_i$.

We can then reduce $M$ to its \emph{Smith Normal Form} (see \cite[Chapter 3]{jacobson}), that is, we can find matrices $P\in \GL_\rho(\Z)$ and $Q\in\GL_t(\Z)$ such that
\begin{align*}
PMQ =\left(\begin{array}{cccccccc}
d_1 & 0 & \cdots & \cdots & \cdots & \cdots & 0 \\
0 & d_2 & &  & & & \vdots\\
\vdots & & \ddots & & & & \vdots\\
\vdots & & & d_r & & & \vdots\\
\vdots & & & & 0& & \vdots \\
\vdots & & & & & \ddots & \vdots \\
0 & \cdots & \cdots & \cdots & \cdots & \cdots & 0
\end{array}\right)
\end{align*}

where $d_1,\dots,d_r$ are integers such that $d_1\mid d_2\mid\cdots\mid d_r$ and $r$ is the rank of $A$. The integers $d_i$ are uniquely determined up to sign, and they are easily computable from the minors of $M$ (see \cite[Theorem 3.9]{jacobson}).

It follows that there is a $\Z$-basis $\{\mathbf{f}_1,\dots, \mathbf{f}_\rho\}$ of $\overline{E(K)}$ such that $\{d_1\mathbf{f}_1,\dots,d_r\mathbf{f}_r\}$ is a $\Z$-basis for $A$. Moreover, if $\Gamma$ is defined as in \eqref{eq:radicalGroup}, we have that $(\Gamma\cap E(K))/E(K)_{\tors}$ is generated by $\{\mathbf{f}_1,\dots,\mathbf{f}_r\}$. We then have that $d_r(\Gamma\cap E(K))\subseteq A+E(K)_{\tors}$, so we can take $d_A=d_r$.

\subsection{The torsion representation}

The torsion representation is nothing but the usual Galois representation attached to the torsion of $E$. After a choice of basis, we will denote it by
\begin{align*}
\tau_\infty:\Gal(K_\infty\mid K)\to \GL_2(\hat\Z)
\end{align*}
and we will denote its image by $H(E)$. If $\ell$ is a prime we will denote by $\tau_{\ell}$ the composition of $\tau_\infty$ with the natural projection $\GL_2(\hat\Z)\to \GL_2(\Z_\ell)$ and by $H_{\ell}(E)$ the image of $\tau_{\ell}$.

\subsubsection{The non-CM case} If $E$ does not have complex multiplication over $\Kbar$,
by Serre's Open Image Theorem (see \cite{serre72}) we know that there exist:
\begin{itemize}
\item an integer $m_E\geq 1$ such that $H(E)$ contains all the elements of $\GL_2(\hat\Z)$ that are congruent to the identity modulo $m_E$ (in particular, $H_{\ell}(E)=\GL_2(\Z_\ell)$ for $\ell\nmid m_E$);
\item for every prime number $\ell$, an integer $n_\ell\geq 1$ such that $H_{\ell}(E)\supseteq I+\ell^{n_\ell}\Mat_{2\times 2}(\Z_\ell)$.
\end{itemize}

\begin{rem}
Notice that, if an explicit bound for $m_E$ is known, one can easily give a bound for each $n_\ell$ by just letting $n_\ell=\max(1,v_\ell(m_E))$. However, it is possible to give an effective bound for each $n_\ell$ (see \cite[Theorem 14 and Remark 15]{lombardoperucca} and \cite[Remark 3.7]{lt}), so we will keep these constants separate.
\end{rem}

\begin{defi}
We call \emph{adelic bound} for the torsion representation a positive even integer $m_E$ such that $H(E)$ contains all the elements of $\GL_2(\hat\Z)$ congruent to the identity modulo $m_E$. If $\ell$ is a prime, we call an integer $n_\ell\geq 1$ such that $H_{\ell}(E)\supseteq I+\ell^{n_\ell}\Mat_{2\times 2}(\Z_\ell)$ a \emph{parameter of maximal growth} for the $\ell$-adic torsion representation. If $\ell=2$ we require $n_\ell\geq 2$.
\end{defi}

\begin{prop}
\label{prop:killH1forEC}
If $m_E$ is an adelic bound for the torsion representation of $E$ over $K$, then $m_EH^1(\Gal(K_\infty\mid K), E_{\tors})=0$. 
\end{prop}
\begin{proof}
  Let  $G=\Gal(K_\infty\mid K)$ and let $z=(z_\ell)_\ell\in \hat\Z=\prod_{\ell}\Z_\ell$ be defined as
\begin{align*}
z_\ell=\begin{cases}
1+\ell^{v_\ell(m_E)} & \text{if } \ell \mid m_E,\\
2 &\text{if }\ell\nmid m_E.
\end{cases}
\end{align*}
Since by definition $2\mid m_E$ we have $z\in\hat\Z^\times$. Moreover $z-1=um_E$ for some $u\in\hat\Z^\times$.

Consider now the element $g=z I\in \GL_2(\hat\Z)$: it is congruent to the identity matrix modulo $m_E$, so it lies in $G$; moreover it is a scalar matrix, so it lies in the center of $G$. By Sah's Lemma (see \cite[Lemma A.2]{baker-ribet}) the endomorphism of $H^1(G,E_{\tors})$ defined by $f\mapsto (g-I)f$ kills $H^1(G, E_{\tors})$. Since $g-I=um_E I$ for $u\in \hat\Z^\times$, we have that $m_EH^1(G,E_{\tors})=0$, as required.
\end{proof}

\begin{defi}
\label{defi:setS}
Let $K$ be a number field with absolute discriminant $\Delta_K$ and let $E$ be an elliptic curve over $K$ without CM over $\Kbar$. We denote by $S(E)$ the finite set of primes $\ell$ that satisfy at least one of the following conditions:
\begin{enumerate}[(1)]
\item $\ell\mid 2\cdot 3\cdot 5 \cdot\Delta_K$;
\item the Galois group $\Gal(K_\ell\mid K)$ is not isomorphic to $\GL_2(\mathbb{F}_\ell)$.
\item $E$ has bad reduction at some prime of $K$ of characteristic $\ell$.
\end{enumerate}
\end{defi}

\begin{rem}
The set $S(E)$ is effectively computable (see \cite[Remark 5.2]{lt}).
\end{rem}

An explicit value for the adelic bound $m_E$ is provided by the following result by F. Campagna and P. Stevenhagen:

\begin{thm}[{\cite[Theorem 3.4]{campstev}}]
\label{thm:CampagnaStevenhagen}
Let $E$ be an elliptic curve over $K$ without CM over $\Kbar$.
Write $K_{\ell^\infty}$ for the compositum of all $\ell$-power division fields of $E$ over $K$, and $K_{S(E)}$ for the compositum of the fields $K_{\ell^\infty}$ with $\ell\in {S(E)}$. Then the family consisting of $K_{S(E)}$ and $\{K_{\ell^\infty}\}_{\ell\not\in {S(E)}}$ is linearly disjoint over $K$, that is, the natural map
\begin{align*}
\Gal(K_\infty\mid K)\to \Gal(K_{S(E)}\mid K)\times\prod_{\ell\not \in {S(E)}}\Gal(K_{\ell^\infty}\mid K)
\end{align*}
is an isomorphism.
\end{thm}

\begin{rem}
\label{rem:elladicSurjOutsideS}
For every prime $\ell\not\in S(E)$, the $\ell$-adic representation associated with $E$ is surjective. This follows from the fact that the $\bmod \,\ell$ torsion representation associated with $E$ and the the $\ell$-adic cyclotomic character of $K$ are both surjective (since $\ell\nmid \Delta_K$): in fact in this case we have $(H(E)\bmod \ell)\supseteq\SL_2(\Z/\ell\Z)$ and $\det(H_{\ell}(E))=\Z_\ell^\times$, which implies (see \cite[IV-23]{serre1998abelian}) that $H_{\ell}(E)=\GL_2(\Z_\ell)$.
\end{rem}

\begin{cor}
\label{cor:boundmE}
For every prime $\ell\in S(E)$ let $n_\ell$ be a parameter of maximal growth for the $\ell$-adic torsion representation. Let moreover $R:=\prod_{\ell\in S(E)}\ell$ and $m_\ell = v_\ell\left([ K_R:K]\right)$.
Then an adelic bound for the torsion representation is given by
\begin{align*}
m_E = \prod_{\ell\in S(E)}\ell^{n_\ell+m_\ell}.
\end{align*}
\end{cor}

\begin{proof}
It will be enough to show that the image of $\Gal(K_\infty\mid K)$ in $\GL_2(\hat\Z)$ contains
\begin{align*}
\prod_{\ell\in S(E)}\left(I+\ell^{m_\ell+n_\ell}\Mat_{2\times 2}(\Z_\ell)\right)\times \prod_{\ell\not\in S(E)} \GL_2(\Z_\ell)\,.
\end{align*}
We will do so by considering the subgroup $\Gal(K_\infty\mid K_R)$ of $\Gal(K_\infty\mid K)$.

Notice that, since for every prime $\ell$ and every $n\geq 1$ the degree of $K_{\ell^n}$ over $K_\ell$ is a power of $\ell$, the family $\{K_{\ell^\infty R}\}_{\ell\in S(E)}$ is linearly disjoint over $K_R$. Then we have
\begin{align*}
\Gal(K_\infty\mid K_R)=& \Gal(K_{S(E)}\mid K_R)\times\prod_{\ell\not \in {S(E)}}\Gal(K_{\ell^\infty}\mid K)=\\
=& \prod_{\ell\in {S(E)}}\Gal(K_{\ell^\infty R}\mid K_R)\times\prod_{\ell\not \in {S(E)}}\Gal(K_{\ell^\infty}\mid K).
\end{align*}
For every $\ell\in S(E)$ we have $\tau_{\ell}(\Gal(K_{\ell^\infty R}\mid K_R))\supseteq I+\ell^{r_\ell}\Mat_{2\times 2}(\Z_\ell)$, where $r_\ell$ is a parameter of maximal growth for the $\ell$-adic torsion representation attached to $E$ over $K_R$. By \cite[Lemma 3.10]{lt} we can take $r_\ell\leq n+m_\ell$, so $\rho_{\infty}(\Gal(K_\infty\mid K_R))$ contains
\begin{align*}
\prod_{\ell\in {S(E)}}\left(I+\ell^{n_\ell+m_\ell}\Mat_{2\times 2}(\Z_\ell)\right)\times\prod_{\ell\not \in {S(E)}}\GL_2(\Z_\ell)
\end{align*}
so it contains all elements that are congruent to $I$ modulo $m_E$, as required.
\end{proof}

\begin{rem}
\label{rem:boundmE}
We can give an explicit bound for the integers $m_\ell$ of the above corollary:
\begin{align*}
m_\ell= v_\ell\left([K_R:K]\right)\leq v_\ell\left(\#\GL_2\left({\Z}/{R\Z}\right)\right)=\sum_{p \in S(E)}v_\ell\left((p^2-1)(p^2-p)\right).
\end{align*}
\end{rem}

\subsubsection{The CM case.}

The torsion representations associated with elliptic curves with complex multiplication have been studied for example in \cite{deur1} and \cite{deur2}. They are deeply related to the endomorphism ring $\mathcal{O}_E=\End_{\Kbar}(E)$ of $E$, which is an order in an imaginary quadratic number field $F$.

For every prime $\ell$, the group
\begin{align*}
\mathcal{C}_{\ell}(E):=\left(\mathcal{O}_E\otimes_{\Z}\Z_\ell\right)^\times
\end{align*}
can be identified with a subgroup of $\GL_2(\Z_\ell)$ via the action of $\mathcal O_E$ on the $\ell$-power torsion of $E$, and is called the \emph{Cartan subgroup of $\GL_2(\Z_\ell)$} associated with $E$. We also let
\begin{align*}
\mathcal{C}(E):=\left(\mathcal{O}_E\otimes_{\Z}\hat\Z\right)^\times=\prod_{\ell\text{ prime}}\mathcal{C}_{\ell}(E)
\end{align*}
which can be identified with a subgroup of $\GL_2(\hat\Z)$, and we denote by $\mathcal{N}_{\ell}(E)$ and $\mathcal{N}(E)$ the normalizers of $\mathcal{C}_{\ell}(E)$ in $\GL_2(\Z_\ell)$ and of $\mathcal{C}(E)$ in $\GL_2(\hat\Z)$, respectively.

The group $\mathcal{C}_\ell(E)$ is always conjugate to a subgroup of $\GL_2(\Z_\ell)$ of the form
\begin{align*}
\left\{ \begin{pmatrix}
x & \delta y \\ y & x+\gamma y
\end{pmatrix} : x,y \in \Z_{\ell},\; v_{\ell}(x(x+\gamma y)-\delta y^2)=0 \right\}
\end{align*}
for some integers $\gamma$ and $\delta$, which are called \emph{parameters} for $\mathcal{C}_\ell(E)$ (see \cite[\S 2.3]{lombardoperucca1eigenspace}).

The image of the torsion representation associated with $E$ is contained in $\mathcal{N}(E)$, and can be described as follows.

\begin{prop}[{\cite[Theorem 1.5]{lombardoGRCM}}]
  \label{prop:CMindex}
Let $E$ be an elliptic curve over $K$ with CM over $\Kbar$, and let $F$ be the CM field of $E$. Let $\mathcal S$ denote the set of primes $\ell$ that either ramify in $K\cdot F$ or are such that $E$ has bad reduction at some prime of $K$ of characteristic $\ell$. Then:
\begin{enumerate}
\item if $F\subseteq K$, then $H(E)\subseteq \mathcal{C}(E)$ and $[\mathcal C(E):H(E)]$ divides $6[K:\Q]$. Moreover, $H_{\ell}(E)=\mathcal{C}_{\ell}(E)$ for every $\ell\not\in \mathcal S$;
\item if $F\not\subseteq K$, then $H(E)\subseteq \mathcal{N}(E)$, but $H(E)\not\subseteq \mathcal C(E)$, and $[\mathcal C(E):\mathcal C(E)\cap H(E)]$ divides $12[K:\Q]$. Moreover, $H_{\ell}(E)=\mathcal{N}_{\ell}(E)$ for every $\ell\not\in \mathcal S$.
\end{enumerate}
\end{prop}

\begin{rem}
  The above mentioned result \cite[Theorem 1.5]{lombardoGRCM} states that
  $[\mathcal C(E):H(E)]\leq 3[K:\Q]$ if $F\subseteq K$ and 
  $[\mathcal C(E):\mathcal C(E)\cap H(E)]\leq 6[K:\Q]$ if $F\not\subseteq K$.
  However, one can check that its proof also yields Proposition
  \ref{prop:CMindex} as stated here.
\end{rem}

\begin{prop}
\label{prop:killH1forECCM}
Let $E$ be a CM elliptic curve over $K$ and let $e_K=12[K:\Q]$. Let moreover
\begin{align*}
m_K:=4^{e_K}\cdot\prod_{\ell}\ell^{e_K},
\end{align*}
where the product runs over all odd primes $\ell$ such that $(\ell-1)$ divides $e_K$.
Then we have $m_K H^1(\Gal(K_\infty\mid K), E_{\tors})=0$. 
\end{prop}
\begin{proof}
  Let $k_2=3$ and, for any odd prime $\ell$, let $k_\ell$ be an integer whose class modulo $\ell$ is a generator of $(\Z/\ell\Z)^\times$ and $1<k_\ell<\ell$.
  Let then $z=(k_\ell^{e_K})_\ell\in\hat\Z$, and let $g=zI\in \GL_2(\hat\Z)$
  By Proposition \ref{prop:CMindex} we have $(\mathcal{C}(E))^{e_K}\subseteq H(E)$, so in particular $g\in H(E)$. Applying Sah's Lemma as in Proposition \ref{prop:killH1forEC} we see that $g-I$ kills $H^1(\Gal(K_\infty\mid K),E_{\tors})$. Since
\begin{align*}
v_2\left(3^{e_K}-1\right)&\leq  2e_K,\\
v_\ell\left(k_\ell^{e_K}-1\right)&\leq e_K \quad \text{for all }\ell>2\,,\\
v_\ell\left(k_\ell^{e_K}-1\right)& = 0 \quad \text{for all }\ell\text{ such that }(\ell-1)\nmid e_K,
\end{align*}
we have that $z-1=um$ for some $u\in\hat\Z^\times$ and some $m$ which divides $m_K$. As in Proposition \ref{prop:killH1forEC} we conclude that the exponent of $H^1(\Gal(K_\infty\mid K), E_{\tors})=0$ divides $m_K$.
\end{proof}

It follows from classical results (see also \cite[Section 2]{lombardoperucca}) that for every prime $\ell$ there is a positive integer $n_\ell$ such that
\begin{align}
\label{eqn:maxgrowCM}
\#(H(E)\bmod{\ell^{n+1}})/\#(H(E)\bmod {\ell^n}) = \ell^2 \qquad \text{for all } n\geq n_\ell\,.
\end{align}

\begin{defi}
We call a positive integer $n_\ell$ satisfying \eqref{eqn:maxgrowCM} a \emph{parameter of maximal growth} for the $\ell$-adic torsion representation. If $\ell=2$ we require $n_\ell\geq 2$.
\end{defi}

\subsection{Main theorems} We can finally prove our main results, which are higher-rank generalizations of \cite[Theorems 1.1 and 1.2]{lt}.

\begin{thm}
\label{thm:mainEC1}
Let $E$ be an elliptic curve over a number field $K$ without complex multiplication over $\Kbar$. Let $A$ be a finitely generated and torsion-free subgroup of $E(K)$ of rank $r>0$.

Let $d_A$ be a divisibility parameter for $A$. Let $S(E)$ be the finite set of primes of Definition \ref{defi:setS} and for every $\ell\in S(E)$ let $n_\ell$ be a parameter of maximal growth for the $\ell$-adic torsion representation of $E/K$ and
\begin{align*}
m_\ell:=\sum_{p\in S(E)}v_\ell((p^2-1)(p^2-p)).
\end{align*}
Then $V(A)$ is open in $\Mat_{r\times 2}(\hat\Z)$. More precisely, the order of $\Ent(A)$ divides
\begin{align*}
\left(d_A\cdot\prod_{\ell\in S(E)}\ell^{2n_\ell+m_\ell}\right)^{2r}.
\end{align*}
\end{thm}
\begin{proof}
  By Remark \ref{rem:elladicSurjOutsideS}, the integer $n:=\prod_{\ell\in S(E)}\ell^{n_\ell}$ satisfies $\Z_\ell\left[H_\ell(E)\right]\supseteq n\Mat_{2\times 2}(\Z_\ell)$ for every prime number $\ell$.
By Corollary \ref{cor:boundmE} and Remark \ref{rem:boundmE} the integer $m:=\prod_{\ell\in S(E)}\ell^{n_\ell+m_\ell}$ is an adelic bound for the torsion representation associated with $E$, so by Proposition \ref{prop:killH1forEC} the exponent of the group $H^1\left(\Gal(K_\infty\mid K), E_{\tors}\right)$ divides $m$.

Then by Theorem \ref{thm:GeneralBoundEntanglement} we have that the order of $\Ent(A)$ divides $(d_Anm)^{2r}$.
\end{proof}

\begin{defi}
\label{def:badPrimesCM}
Let $E$ be an elliptic curve over a number field $K$ with CM over $\Kbar$. Let
$\mathcal{O}_E=\End_{\Kbar}(E)$ and let $F=\Frac(\mathcal{O}_E)$. We denote by
$S(E)$ the finite set of primes such that at least
one of the following conditions is satisfied:
\begin{enumerate}
\item $\ell$ divides the conductor of $\mathcal{O}_E$;
\item $\ell$ ramifies in $K\cdot F$;
\item $E$ has bad reduction at some prime of $K$ of characteristic $\ell$.
\end{enumerate}
\end{defi}

\begin{thm}
\label{thm:mainEC1cm}
Let $E$ be an elliptic curve over a number field $K$, with CM over $\Kbar$ but not over $K$. Let $A$ be a finitely generated and torsion-free subgroup of $E(K)$ of rank $r>0$.

Let $d_A$ be a divisibility parameter for $A$. For every prime $\ell$ let $n_\ell$ be a parameter of maximal growth for the $\ell$-adic torsion representation of $E/K$ and let $(\gamma_\ell,\delta_\ell)$ be parameters for $\mathcal{C}_\ell(E)$. Let $m_K$ be the integer defined in Proposition \ref{prop:killH1forECCM}. Let moreover $S(E)$ be the finite set of primes of Definition \ref{def:badPrimesCM}.

Then $V(A)$ is open in $\Mat_{r\times 2}(\hat\Z)$. More precisely, the order of $\Ent(A)$ divides
\begin{align*}
\left(d_Am_K\cdot \prod_{\ell\in S(E)}\ell^{n_\ell+v_\ell(4\delta_\ell)}\right)^{2r},
\end{align*}
where we let $v_\ell(0)=0$ for every prime $\ell$.
\end{thm}

\begin{proof}
In order to apply Theorem \ref{thm:GeneralBoundEntanglement} we only need to prove that:
\begin{enumerate}
\item for every prime $\ell\not\in S(E)$ we have
\begin{align*}
\Z_\ell[H_{\ell}(E)] = \Mat_{2\times 2}(\Z_\ell)\,;
\end{align*}
\item for every prime $\ell \in S(E)$ we have
\begin{align*}
\Z_\ell[H_{\ell}(E)]\supseteq \ell^{n_\ell+v_\ell(4\delta_\ell)}\Mat_{2\times 2}(\Z_\ell)\,.
\end{align*}
\end{enumerate}
Both parts follow from from \cite[Proposition 4.12, proof of (3)]{lt}, noticing that for every $\ell\not\in S(E)$ one may take $d=0$ by \cite[Proposition 10]{lombardoperucca1eigenspace}.
\end{proof}

\begin{thm}
\label{thm:mainEC2}
There is a universal constant $C\geq 1$ such that, for every elliptic curve $E/\Q$ and every torsion-free subgroup $A$ of $E(\Q)$, the order of $\Ent(A)$ divides $(d_AC)^{2\rk(A)}$.
\end{thm}
\begin{proof}
By \cite[Corollary 3.13]{lt} (which relies on \cite[Theorem 1.2]{arai} for the non-CM case) the parameters of maximal growth for the $\ell$-adic torsion representation associated with an elliptic curve over $\Q$ can be bounded independently of $E$. By \cite[Theorem 1.3]{lt} there is a constant $C_1$ such that the exponent of $H^1(\Gal(\Q_\infty\mid \Q),E_{\tors})$ divides $C_1$. The conclusion then follows from Theorem \ref{thm:GeneralBoundEntanglement}.
\end{proof}

\begin{rem}
Theorem \ref{thm:mainEC1cm} does not hold if $\mathcal{O}_E=\End_K(E)\neq \mathbb{Z}$.
In fact in this case one may find a subgroup $A\subseteq E(K)$ such that $\Ent(A)$
is infinite.

To see this, let $P\in E(K)$ be a point of infinite order and let
$A=\mathcal{O}_E P$ and $A'=\Z P$. Since $A$ is a free $\mathcal{O}_E$-module
of rank $1$, it has rank $2$ as an abelian group.

Let $Q\in n^{-1}P$. For every $n>1$ and every $\sigma\in\mathcal{O}_E$
we have $n^{-1}\sigma(P)=\sigma(Q)+E[n]$, so 
\begin{align*}
  n^{-1}A=\mathcal{O}_EQ+E[n].
\end{align*}
Since $Q\in n^{-1}A'$ and $\mathcal O_E$ is defined over $K$ we have that 
$\mathcal O_EQ$ is defined over $K(n^{-1}A')$.
Since moreover $E[n]\subseteq n^{-1}A'$ we deduce that 
$K(n^{-1}A)\subseteq K(n^{-1}A')$. In fact, since $A\supseteq A'$,
the two fields coincide. So in particular
\begin{align*}
\left[K\left(n^{-1}A\right):K\left(E[n]\right)\right] = \left[K\left(n^{-1}A'\right):K\left(E[n]\right)\right]\,.
\end{align*}
Then for every $n>1$ we have by Remark \ref{rem:kummerDegDividesN}
\begin{align*}
\frac{n^{4}}{\left[K\left(n^{-1}A\right):K\left(E[n]\right)\right]}= \frac{n^{4}}{\left[K\left(n^{-1}A'\right):K\left(E[n]\right)\right]}\geq n^2
\end{align*}
which, by Lemma \ref{lemma:entBoundsDegrees}, implies that $\Ent(A)$ is infinite.

Notice that the points we consider in this example are not linearly independent
over $\mathcal O$. In fact, the condition that the points are lineraly
independent over the endomorphism ring of the curve can also be found in
\cite[Theorem 1.2]{ribet}.
\end{rem}

\bibliographystyle{acm}

\end{document}